\documentclass{amsart}

\usepackage{mathrsfs}
\usepackage{amsmath,amssymb,enumerate,tikz}
\usepackage{amsfonts,amssymb}
\usepackage{pinlabel}
\usepackage{latexsym,amsfonts,amssymb,verbatim,mathrsfs,amsthm}
\usepackage{amsmath,amsthm,amssymb,latexsym,graphics,textcomp,graphicx,subfigure}
\usepackage{eucal,eufrak}
\usepackage{colonequals}
\usepackage{graphicx}
\usepackage{url}
\usepackage[T1]{fontenc}

\newtheorem{theorem}{\rm \bf Theorem}[section]
\newtheorem{proposition}[theorem]{  \rm \bf Proposition}
\newtheorem{corollary}[theorem]{\rm \bf Corollary}
\newtheorem{lemma}[theorem]{\rm \bf Lemma}
\newtheorem*{mtheorem}{Main Theorem }

\theoremstyle{definition}
\newtheorem{remark}{Remark}

\newtheorem{defi}{Definition}

\newcommand{\Teich}{\ Teichm\"uller space\ }
\renewcommand{\remark}{\textbf{Remark}\quad   }

\newcommand{\Type}{\textbf{Type} }
\newcommand{\ML}{\mathcal{ML}(S)}

\newcommand{\R}{\mathbb{R}}
\newcommand{\C}{\mathcal{C}}

\title{ the Hilbert metric on Teichm\"uller space and  Earthquake}

\date{October 10, 2015}
\author{Huiping Pan}

\address{School of Mathematics and Computational Science, Sun Yat-Sen University, 510275, Guangzhou, P. R. China}
\email{ chnpanhp@foxmail.com}
 \thanks{The
  author is partially supported by NSFC, No: 11271378.}

\begin{document}

\synctex=1
\begin{abstract}
Hamenst\"adt  gave a  parametrization of the Teichm\"uller space of punctured surfaces  such that the image under this parametrization is the interior of a polytope.  In this paper, we study the Hilbert metric on the Teichm\"uller space of punctured surfaces based on this parametrization. We prove that every earthquake ray is an almost geodesic under the Hilbert metric.
\bigskip

\medskip

\end{abstract}
\maketitle
\section{introduction}
Let $S_{g,n}$ be an orientable surface of  genus $g$ with $n$ punctures. In this paper, we  consider the surfaces of negative Euler characteristic with at least one puncture. A marked hyperbolic structure on $S_{g,n}$ is pair $(X,f)$ where $X$ is a complete hyperbolic metric on a surface $S$ and $f:S_{g,n}\to S$ is a homeomorphism. Two marked hyperbolic structures $(X_1,f_1)$ and $(X_1,f_2)$ are called equivalent if there is an isometry in the isotopy class of $f_1^{-1}\circ f_2$. For simplicity, we usually denote a marked hyperbolic metric by $X$ instead of the pair $(X,f)$.   The Teichm\"uller space $T_{g,n}$ is defined as the space of  equivalent classes of marked hyperbolic structures on $S_{g,n}$. It is well known that $T_{g,n}$, equipped with the natural topology is homeomorphic to a ball in $\mathbb R^{6g-6+2n}$.

 Given an open  convex domain $D\subset \mathbb R^m$, Hilbert defined a natural metric on $D$, now called the Hilbert metric,  such that the straight line segments are geodesic segments under this metric.    In ~\cite{Papa2}, Papadopoulos raised a problem: \lq\lq Realize Teichm\"uller space as a bounded convex set somewhere and study the Hilbert metric on it \rq\rq. In this paper, we will study this problem.

 For the case of closed surfaces,   Yamada \cite{Yamada1} constructed a space which he called the Teichm\"uller-Coxeter complex   within which the original Teichm\"uller space sits as an open convex but unbounded subset,  based on the Weil-Petersson completion of the Teichm\"uller space. Then in \cite{Yamada2}, after  introducing a new variational characterization  of the Hilbert metric,  he defined  the Wei-Petersson Hilbert metric on the Teichm\"uller space, where the background geometry is the one induced by the Weil-Petersson geometry instead of the Euclidean geometry.

 For the case of punctured surfaces, Hamenst\"adt \cite{Ham1} provided a parametrization of the Teichm\"uller space into $\mathbb RP^{6g-6+2n}$ by length functions such that the image of $T_{g,n}$ is the interior of a finite-sided polyhedron (see \S\ref{sec:geometri para}). Therefore the Hilbert metric is well defined on $T_{g,n}$. Hamenst\"adt's parametrization depends on the choice of a preferred triangulation $\Gamma$ of $S_{g,n}$. More precisely, fix a puncture of $S_{g,n}$ and denote it as $O$, let $\Gamma=\{\eta_1,\eta_2,...,\eta_{6g-5+2n}\}$ be a set of bi-infinite simple curves on $S_{g,n}$ such that for any marked hyperbolic metric $X$ the two ends of $\eta_i$, $i=1,2,...,6g-5+2n$, go into the puncture $O$ and such that $S_{g,n}\backslash \Gamma$ consists of $4g-3+n$ ideal triangles and $n-1$ once punctured discs. Such a set $\Gamma$ is called a \textit{preferred triangulation} of $S_{g,n}$. There are countably many choices of preferred triangulations.

 In this paper, we study the Hilbert metric $d^{\Gamma}_h$ on $T_{g,n}$ based on Hament\"adt's parametrization.
 Before stating our main result, we briefly explain a deformation of hyperbolic metric introduced by Thurston in \cite{Th1}, namely, the \textit{earthquake}.

 Let $\alpha$ be a simple closed curve on $S_{g,n}$, and $X\in T_{g,n}$ be a marked hyperbolic metric.  Denote by $\alpha^*$ the  geodesic representative of  $\alpha$ on $X$.  Cutting $X$ along $\alpha^*$ and twisting to the left about distance $t$, we obtain a new marked hyperbolic metric, denoted by $\mathcal E^t_\alpha X$. Note that the notion of ``left'' twist depend only on the orientation of $X$ (no orientation of $\alpha^*$ is necessary).
 Thurston extended this construction to any measured geodesic lamination. He proved the following result, one of whose proof can be found in \cite{Ke}.

\begin{proposition}\label{prop:earthquake}
  There is a (unique) continuous map $\mathcal{ML}\times \mathbb R \times T_{g,n}\to T_{g,n}$, associating an element $\mathcal E^t_\alpha X\in T_{g,n}$ to $(\alpha,t,X)$, such that $\mathcal E^t_{\lambda\alpha}X=\mathcal E^{\lambda t}_\alpha X$ for all $\lambda>0$ and all $\alpha\in \mathcal {ML}$, and such that when $\alpha$ is a simple closed geodesic, $\mathcal E^t_\alpha X$ is obtained from $X$ by the earthquake defined above.
\end{proposition}

The metric $\mathcal E^t_\alpha X$ defined in Proposition \ref{prop:earthquake} is said to be obtained from $X$ by a \textit{(left) earthquake of amplitude t along the measured geodesic lamination $\alpha$}, and the orbits $\{\mathcal E^t_\alpha X\}_{t=-\infty}^\infty$, $\{\mathcal E^t_\alpha X\}_{t=-\infty}^0$ and $\{\mathcal E^t_\alpha X\}_{t=0}^\infty$ are called the \textit{earthquake line directed by $\alpha$ and starting at $X$}, the \textit{anti-earthquake ray directed by $\alpha$ and starting at $X$}, and  the \textit{earthquake ray directed by $\alpha$ and starting at $X$}, respectively.

 Recall that for a metric space $(X,d)$, an unbounded path $\gamma:[0,\infty)\to X$ is called an \textit{almost-geodesic} if for any $\epsilon>0$,  there exists $T>0$, such that
 $$ |d(\gamma(0),\gamma(s))+d(\gamma(s),\gamma(t))-t)|< \epsilon$$
 for any $t\geq s\geq T.$

 Now we  state our main result.
  \begin{mtheorem}\label{thm:main-theorem}
  After reparametrization, every (anti-)earthquake ray is an almost-geodesic in $(T_{g,n},d^{\Gamma}_h)$.
\end{mtheorem}
 In fact, the image of an earthquake ray under Hamenst\"adt's parametrization eventually looks like a projective line (see \S~\ref{sec:earthquake}).

 \vskip 10pt
 \subsection*{Outline}
 This paper is organized as the following. In Section 2, we recall some basic properties of  the Hilbert metric  and express the Hilbert  metric $d^\Gamma_h$  on the Teichm\"uller space based on Hamenst\"adt's parametrization.
 In Section 3, we prove our main theorem.
 In Section 4, we study the  dependence of the Hilbert metric $d^{\Gamma}_h$ on the choice of the preferred triangulation $\Gamma$. We will show that a sphere $B(X_0,R)$   centered at $X_0\in T_{g,n}$ of radius $R$ with respect to  $d^\Gamma_h$ for a preferred triangulation $\Gamma$ is again a sphere up to an additive constant with respect to $d^{\Gamma'}_h$ for another preferred triangulation $\Gamma'$, provided that $\Gamma'$ can be obtained from $\Gamma$ by a \textit{diagonal-flip} (to be defined in \S~\ref{sec:triangulation}). But the additive constant depends on the center point $X_0$.
 In Section 5, we study the actions of  the mapping class group on $(T_{g,n},d^{\Gamma}_h)$. It is well known that when the Teichm\"uller space is endowed with the Teichm\"uller metric, the Thurston metric or the Weil-Petersson metric, the mapping class group acts by isometries. But here, we will show that  the action of a positive Dehn twist is not isometric (see Corollary \ref{cor:nonisometry}). Instead, it is an almost isometry on an unbounded subset of $(T_{g,n},d^{\Gamma}_h)$ (see Corollary \ref{cor:Dehn-twist}).

\vskip 20pt
\textbf{Acknowledgements.}  I would like to thank my advisor Lixin Liu for his careful reading of this manuscript and many useful suggestions. I would like to thank Athanas Papadopoulos for useful conversations during his visit at Sun Yat-Sen University (Zhongshan University). I also thank the referee for numerous comments and suggestions.

\vskip 40pt
\section{the Hilbert metric on the Teichm\"uller space}\label{sec:Hilbert}

 \subsection{The Hilbert metric}\label{subsec:Hilbert}
 There are two versions of the Hilbert's metric. The first version is the original one due to Hilbert which is  defined  on a bounded convex domain $\Omega$ (see Fig.\ref{fig:hilbert}) in $\mathbb R^m$. Let $x, y $ be two points in the interior of $\Omega$, the line passing through $x, y$ intersects the boundary  $\partial{\Omega}$ at two points $a, b$, where $x$ lies between $a$ and $y$. Then the Hilbert metric is defined as:
\begin{equation}\label{eq:Hilbert}
 d_{H}(x,y)=\frac{1}{2}\log{[a,b,y,x]}=\frac{1}{2}\log{\frac{|a-y||b-x|}{|a-x||b-y|}},
\end{equation}
where $[a,b,y,x]$ represents the cross-ratio of $a,\ x,\ y,\ b$.
\par
The second version is due to Birkhoff which is defined  on the cone $\C$ over a bounded convex domain  $\Omega$ (see \cite{KN}, \cite{KP} and \cite{LN} for more details about this version). Recall that a cone $\C$  is called \textit{pointed} if $\C \cap -\C={0}$. Let $\C$ be a closed, pointed (convex) cone over a convex bounded domain $\Omega$ in $\mathbb R^m$.  Given two nonzero vectors $x$ and $y$ in $\C$ (see Fig.\ref{fig:birk-hilbert}), the \textit{Birkhoff's version of the Hilbert metric}, denoted as $d_h$ is defined as:
\begin{equation}\label{eq:Birkhoff}
  d_{h}(x, y)=\frac{1}{2}\log{M(x,y)/m(x,y)},
\end{equation}
 where
 \[M(x,y)=\inf\{\lambda \geq0:\lambda y-x\in \C\},\]
 \[m(x,y)=\sup\{\lambda \geq 0: x-\lambda y\in \C\}.\]
 Denote by $o$ the cone point of $\mathcal C$, and suppose that the line $\overline{xy}$ passing through $x,y$ intersects the boundary $\partial \mathcal C$ at $a,b$.
 To calculate $M(x,y)$ explicitly, we distinguish two cases. The first case is that the points $o,y,x$ are collinear. In this case $M(x,y)=m(x,y)=|x|/|y|$, hence $d_h(x,y)=0$. The second case is that the points $o,x,y$ are not collinear. We draw an auxiliary line $\overline{xp}$ from $x$ which is parallel to the line $\overline{ob}$ and intersects the line $\overline{oy}$ at $p$. Then
  \[M(x,y)=\frac{| p-o |}{| y-o |}=\frac{| x-b | }{| y-b |}.\]
  Similarly we get
   \[m(x,y)=\frac{|x-a|}{| y-a |}.\]
  Hence
  \[\frac{M(x,y)}{m(x,y)}=\frac{| a-y| | b-x |}{| a-x||b-y |}=[a,b,y,x].\]
  By the property of cross-ratio, we have $d_{h}(\lambda x, \mu y)=d_{h}(x,y)$  for any $\lambda>0, \mu>0$.
  It is clear that $d_h$ is not a metric on $\mathcal C$ since it does not separate $x$ and $\lambda x$ for any $\lambda>0$. In fact, $d_h$ is a metric on the projective space $\mathcal C/ \mathbb R^+$.

 \begin{figure}
   \subfigure[]
   {
   \begin{minipage}[tbp]{50mm}
    \includegraphics[width=50mm]{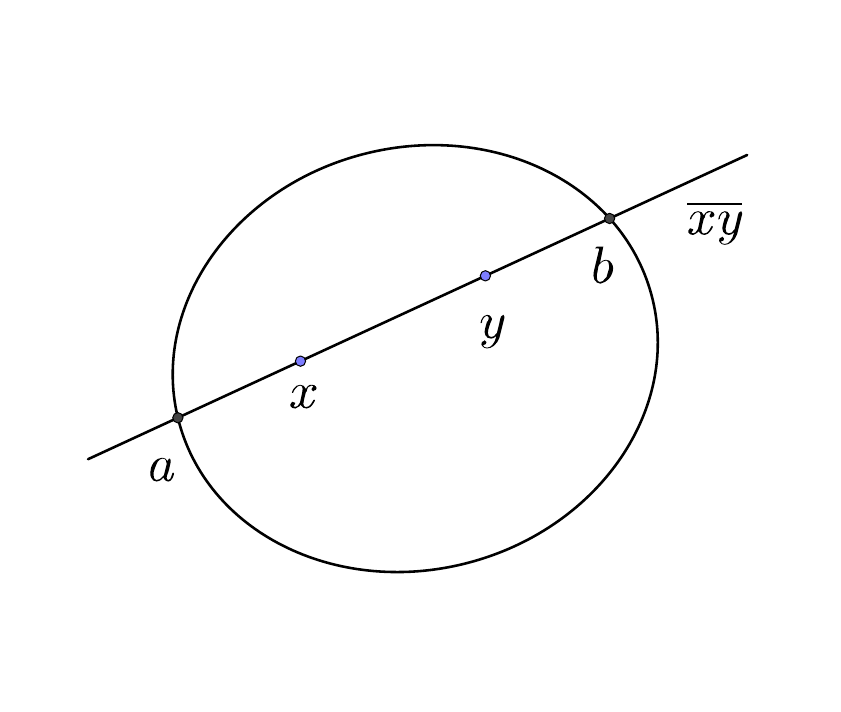}
   \end{minipage}
   \label{fig:hilbert}
   }
      \subfigure[]
   {
   \begin{minipage}[tbp]{50mm}
    \includegraphics[width=50mm]{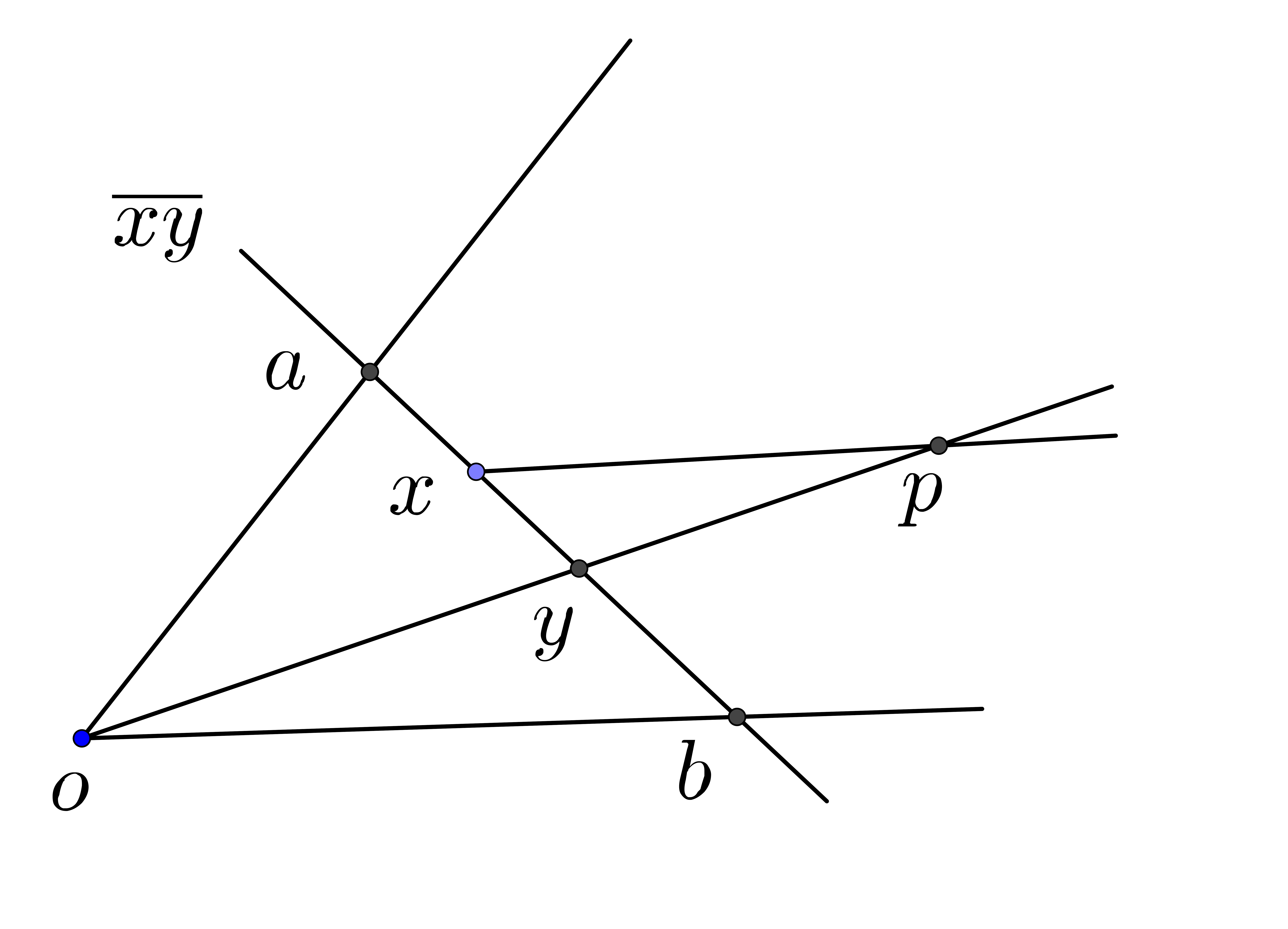}
   \end{minipage}
   \label{fig:birk-hilbert}
   }

      \subfigure[]
   {
   \begin{minipage}[tbp]{50mm}
    \includegraphics[width=50mm]{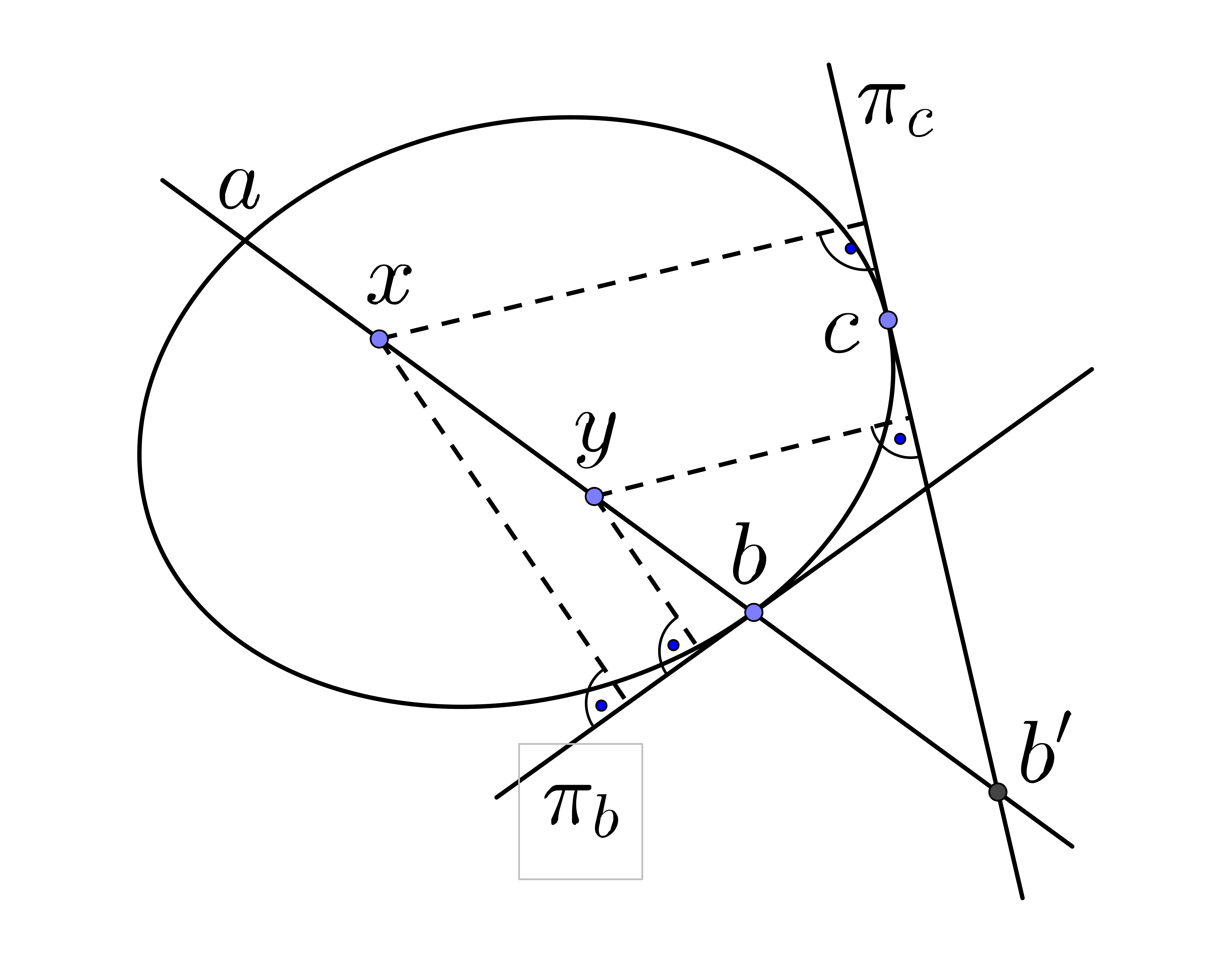}
   \end{minipage}
   \label{fig:yamada-hilbert}
   }
  \caption{the Hilbert metric}
 \end{figure}

\vskip 10pt
  Yamada \cite{Yamada2} gave an alternate definition of the Hilbert metric by supporting hyperplanes. Recall that a convex set $\Omega$ can be represented as $\cap_{\pi(b)\in \mathcal{P}}H_{\pi(b)}$ where $H_{\pi(b)}$ is the half space bounded by a supporting hyperplane $\pi(b)$ of $\Omega$ at the boundary point b, containing the convex set $\Omega$. Let $\mathcal{P}$ be the set of all the supporting hyperplanes of $\Omega$ (see Fig.\ref{fig:yamada-hilbert}). Yamada showed  that the Hilbert metric can be represented as:
 \begin{equation}
   d_{H}(x,y)=\frac{1}{2}(\sup_{\pi \in \mathcal{P}}\log \frac{d(x,\pi)}{d(y,\pi)}+\sup_{\pi \in \mathcal{P}}\log \frac{d(y,\pi)}{d(x,\pi)}).
 \end{equation}

We briefly explain Yamada's idea here. Let the line $\overline{xy}$ intersects $\Omega$ at $a,b$ in the order $a,x,y,b$. Let $\pi(b)$ be a supporting hyperplane of $\Omega$ at $b$, and let $\pi(c)$ be a hyperplane of $\Omega$ at an arbitrary point $c\in\Omega$. Denote by $b'$ the intersection point between $\overline{xy}$ and $\pi(c)$. It is clear that
$$\frac{d(x,\pi(c))}{d(y,\pi(c))}=
\frac{|x-b'|}{|y-b'|}\leq \frac{|x-b|}{|y-b|}
=\frac{d(x,\pi(b))}{d(y,\pi(b))},
\text{ for any }c\in \Omega. $$
Similarly, we have
$$\frac{d(y,\pi(c))}{d(x,\pi(c))}
\leq\frac{d(y,\pi(a))}{d(x,\pi(a))}
\text{, for any }c\in \Omega. $$
Since a cone over a bounded convex set is again a convex set, Yamada's idea also applies to the Birkhoff's version of the Hilbert metric. Therefore
\begin{equation}\label{eq:Yamada}
  d_{h}(x,y)=\frac{1}{2}(\sup_{\pi \in \mathcal{P}}\log \frac{d(x,\pi)}{d(y,\pi)}+\sup_{\pi \in \mathcal{P}}\log \frac{d(y,\pi)}{d(x,\pi)}),\text{ for any }x,y\in\mathcal C\ \text{ with } [x]\neq[y],
 \end{equation}
where $[x]=\{\lambda x\in\mathcal C: \lambda>0\},$ and $\mathcal P$ is the set of supporting hyperplanes of the cone $\mathcal C$.

 \vskip 10pt
  \remark It is easy to see that these two versions of Hilbert metric coincide on the convex bounded domain $\Omega$, i.e.  $d_H(x,y)=d_h(x,y)$  for any $ x,y \in \Omega$. 
  In this paper, we adopt the Birkhoff's version of the Hilbert metric.

\vskip 10pt

 \subsection{Measured geodesic lamination}\label{sec:lamination}

 Given a reference metric $X_0$, a \textit{geodesic lamination} $L$ is a closed subset of $S_{g,n}$ consisting of mutually disjoint simple geodesics which are called \textit{leaves} of this geodesic lamination. A \textit{transverse invariant measure} $\mu$ of a geodesic lamination  $L$ is a Radon measure defined on  every arc $k$ transverse to the support of $L$ such that $\mu$ is invariant with respect to any homotopy of $k$ relative to the leaves of $L$.
 A \textit{measured geodesic lamination} is a lamination $L$ endowed with a transverse invariant measure $\mu$. The simplest example of a measured geodesic lamination is a simple closed geodesic, where the transverse invariant measure is the Dirac measure.   Each measured geodesic lamination $\mu$ induces a functional on the space $\mathcal{S}$ of isotopy classes of nontrival simple closed curves on $S_{g,n}$, which assigns $\inf_{\tilde\gamma\in[\gamma]}\int _{\tilde \gamma}d\mu$ to any $[\gamma]\in \mathcal{S}$. The amount $\inf_{\tilde\gamma\in[\gamma]}\int _{\tilde \gamma}d\mu$ is called the \textit{intersection number} of $\mu$ with $[\gamma]$ and is denoted by $i(\mu,[\gamma])$. Two measured geodesic laminations $\mu,\mu'$ are called \textit{equivalent} if $i(\mu,[\gamma])=i(\mu',[\gamma])$ for any $[\gamma]\in \mathcal S$. Denote by $\mathcal {ML}$ the space of equivalent classes of measured geodesic laminations on $S_{g,n}$, and equip $\mathcal{ML}$  with the weak topology of the functional space over $\mathcal S$. With this topology, the set of weighted simple closed curves , $\mathbb R^+\times \mathcal S$,  is dense in $\mathcal {ML}$.
 The Thurston boundary $\partial T_{g,n}$ of the Teichm\"uller space $T_{g,n}$ consists of the projective classes of measured geodesic laminations and is homeomorphic to the unit sphere $\mathbb S^{6g-6+2n}$ (for more details about measured geodesic laminations we refer to \cite{FLP} and \cite{Penner}).

 \vskip 10pt
 \subsection{Geometric parametrization of the Teichm\"uller space}\label{sec:geometri para}
 Hamenst\"adt  \cite{Ham1}  gave a geometric parametrization of the \Teich  $T_{g,m}$. We briefly recall this parametrization. Let $n\geq1$ and let $X\in T_{g,n}$.
 Fix one of the punctures of $X$ and denote it by $O$. As we explained in the introduction, a preferred triangulation $\Gamma$ is a set of  $6g-5+2n$  mutually disjoint simple geodesics ${\eta}_1, ...,{\eta}_{6g-5+2n}$ on $S$ with the  two ends of each simple geodesic going into the puncture $O$ and which decompose $X$ into $4g-3+n$ ideal triangles and $n-1$ once-punctured discs.   The space of measured lamination $\ML$ on $X$ can be parameterized by the $6g-5+2n$-tuple $(i({\eta}_1,\mu), ...,i({\eta}_{6g-5+2n},\mu))\in \R^{6g-5+2n}$, where $\mu$ is a measured geodesic lamination with compact support and $i(\eta_i,\mu)$ represents the intersection number of $\mu$ with $\eta_i$. Let $\mathcal{A}$ be the set of all $6g-5+2n$-tuples$ (a_1,...,a_{6g-5+2n})$ of nonnegative real numbers with the following properties:
 \par(1) $a_i\leq a_j + a_k$ if the geodesics $\eta_i, \eta_j,\eta_k$ are the sides of an ideal triangle on $S$.
 \par(2) There is at least one ideal triangle on $S$ with sides $\eta_i, \eta_j,\eta_k$  such that $a_i= a_j + a_k$.

 In particular, $\mathcal{A}$ is a cone with vertex at the origin over the boundary of a convex finite-sided polyhedron $P$ in the sphere $\mathbb S^{6g-6+2n}$. And $\mathcal{A}$ is homeomorphic to $\R^{6g-6+2n}$.

 \begin{theorem}[\cite{Ham1}]\label{MF}
 The map
 \par $\mu \in \ML \to (i({\eta}_1,\mu), ...,i({\eta}_{6g-5+2n},\mu))\in \R^{6g-5+2n}$ is a homeomorphism of $\ML$ onto $\mathcal{A}$.
 \end{theorem}

 Since for any marked hyperbolic metric $X\in T_{g,m}$, the length of $\eta_i$ is infinite for any $i=1,2,...,6g-5+2n$,  we need a little modification to parameterize the \Teich   $T_{g,m}$. Recall that every puncture of $S$ admits a standard cusp neighbourhood which is isometric to a cylinder $[-\log2,\infty)\times S^1$ equipped with the metric $d\rho^2+e^{-2\rho}dt^2$ (\cite{Bu}). The $\rho$-coordinate in this representation is called the height. Let $\Delta_{\infty}$ be an ideal triangle on the upper half plane (see Fig.~\ref{fig:triangle}) with sides $\eta_1,\eta_2,\eta_3$. Each corner of $\Delta_{\infty}$ can be foliated by horocycles. Extend these foliations until they fill all but in the center bounded by three horocycles $M_{1}M_{2},M_{2}M_{3},M_{3}M_{2}$. We call $M_i$ the \textit{midpoint} of $\eta_i$ with respect to the ideal triangle $\Delta_\infty$ for $i=1,2,3$. Choose the height $\rho_0$ small enough such  that $e^{-\rho_0}$ is much smaller than  the length of the horocycle $M_{1}M_{2}$ (note that the three horocycles $M_{1}M_{2},M_{2}M_{3},M_{3}M_{2} $ have the same length).
Every geodesic going into the cusp meets the horocycles $\rho$=constant orthogonally. Hence each choice of a height $\rho_0$ cuts from $\eta_i$ a unique compact arc of finite length since both ends of the geodesic $\eta_i$ go into the cusp for any $i=1,2,...,6g-5+2n$. Denote by $l_{\eta_i}(X)$ the length of this subarc of $\eta_i$ for a marked hyperbolic metric $X$ for any $i=1,2,...,6g-5+2n$.
\begin{theorem}[\cite{Ham1}]\label{thm:Teich}
Let $\Pi:\mathbb R^{6g-5+2n}\backslash \{0\}\to \mathbb R P^{6g-6+2n}$ be the canonical projection.
The map
$$\Lambda: S\in T_{g,n} \to (l_{\eta_1}(S),...,l_{\eta_{6g-5+2n}}(S))\in \R^{6g-5+2n}$$
is a diffeomorphism of $T_{g,n}$ onto a hypersurface in $\R^{6g-5+2n}$. And the map $\Pi\circ \Lambda$ is a diffeomorphism of $T_{g,n}$ onto the interior of a finite-sided closed convex polyhedron $P$ in $\mathbb{R}{P}^{6g-6+2n}$ which extends to a homeomorphism of $T_{g,n}\cup \partial{T_{g,n}}$ onto $P$.
\end{theorem}
\begin{figure}
  \includegraphics[width=70mm]{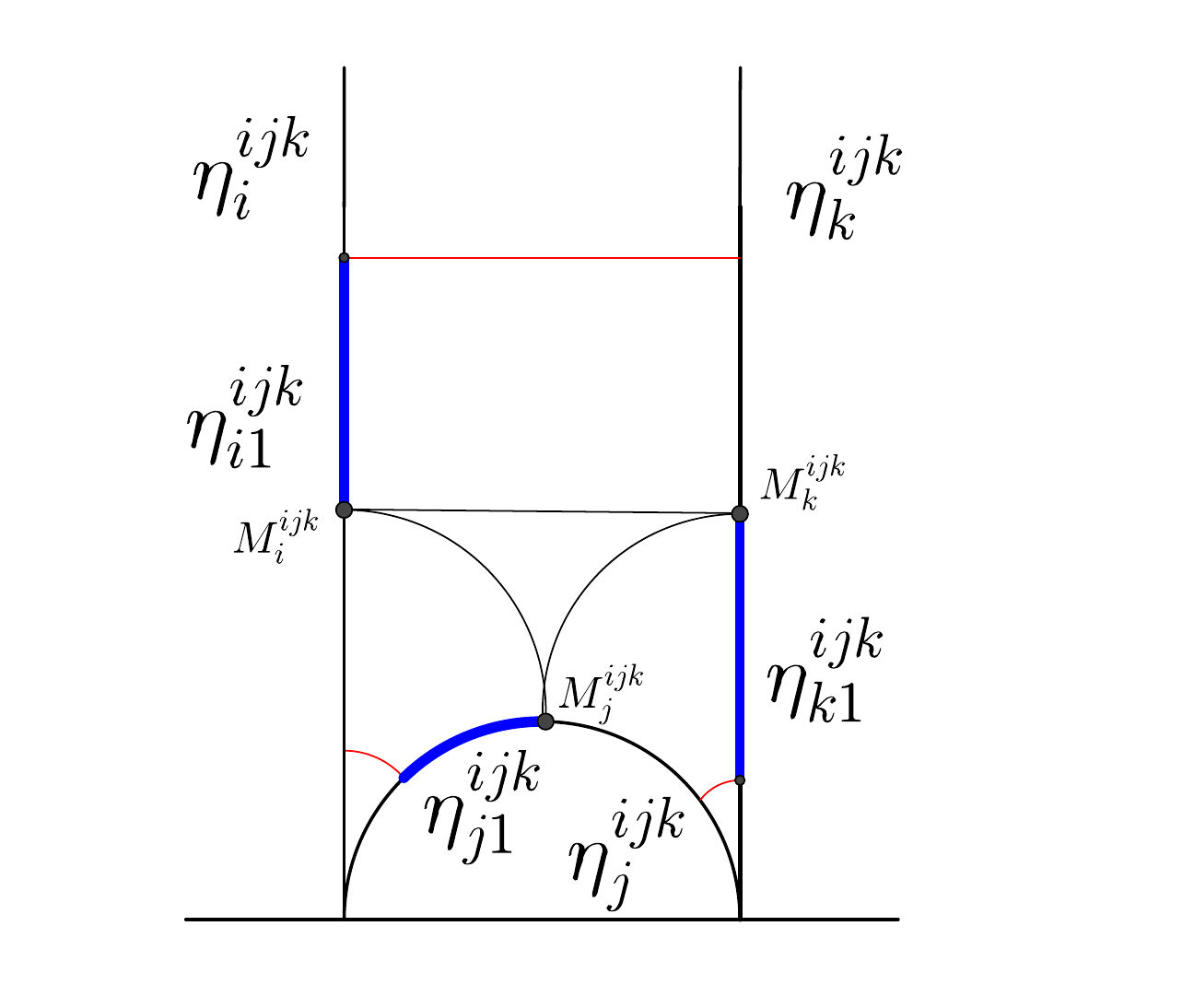}

  \caption{The ideal triangle $\Delta_\infty$}
  \label{fig:triangle}
\end{figure}

\subsection{The Hilbert  metric on the Teichm\"uller space}
Now we give the Hilbert  metric on the Teichm\"uller space. Let $\pi_{ijk}$ be the hyperplane in $\R^{6g-5+2n}=\{(a_1,a_2,...,a_{6g-5+2n}): a_i \in \R\}$ defined by
\begin{equation}\label{pi}
 \pi_{ijk}:a_i-a_j+a_k=0,
\end{equation}
where the indices  $i,\ j,\ k$ satisfy the condition that the corresponding simple geodesics $\eta_i,\ \eta_j,\ \eta_k$ bound an ideal triangle on the surface. Let $\mathcal{P}$ be the set of all such hyperplanes.
By Theorem \ref{MF} and Theorem \ref{thm:Teich}, we know that the image of $T_{g,n} \cup \partial{T_{g,n}}$ under the map $\Pi\circ\Lambda$ is  a polyhedron $P$. 
Following Yamada's idea, we define the Hilbert metric on $T_{g,n}$ by ~(\ref{eq:Yamada}):
\begin{eqnarray}
   {d}^{\Gamma}_h(X_1,X_2)&:=&d_h(\Lambda(X_1),\Lambda(X_2))\nonumber\\
   &=&\frac{1}{2}(\sup_{\pi\in \mathcal{P}}\frac{d(\Lambda(X_1),\pi_{ijk})}{d(\Lambda(X_2),\pi_{ijk})}+\sup_{\pi\in \mathcal{P}}\frac{d(\Lambda(X_2),\pi_{ijk})}{d(\Lambda(X_1),\pi_{ijk})}),
 \end{eqnarray}
for all $X_1,X_2\in T_{g,n}$, where $d(\ ,\ )$ represents the Euclidean distance.
From  (\ref{pi}), we have
\[d(\Lambda(X),\pi_{ijk})=\frac{l_{\eta_i}(X)-l_{\eta_j}(X)+l_{\eta_i}(X)}{3^{1/2}}.\]
Hence
 \begin{eqnarray}\label{eq:def}
  {d}^{\Gamma}_h(X_1,X_2)&=&\frac{1}{2}(\sup_{\pi_{ijk}\in \mathcal{P}}\log\frac{{l_{\eta_i}(X_1)-l_{\eta_j}(X_1)+l_{\eta_k}(X_1)}}
{{l_{\eta_i}(X_2)-l_{\eta_j}(X_2)+l_{\eta_k}(X_2)}}\\
&&
+\sup_{\pi_{ijk}\in \mathcal{P}}\log\frac{{l_{\eta_i}(X_2)
-l_{\eta_j}(X_2)+l_{\eta_k}(X_2)}}{{l_{\eta_i}
(X_1)-l_{\eta_j}(X_1)+l_{\eta_k}(X_1)}})\nonumber.
 \end{eqnarray}

 Now we give some explanations for ~(\ref{eq:def}). Since the surface $S_{g,n}$ is orientable, we fix an orientation. Denote by $\Delta_{ijk}$ the ideal triangle with three sides $\{\eta_i,\eta_j,\eta_k\}$ such that $\eta_i,\eta_j,\eta_k$ appears in the counterclockwise order. Obviously $\Delta_{ijk}=\Delta_{jki}=\Delta_{kij}$.  Let $\mathbb T_\Gamma$ be  the collection of all such triangles corresponding to $\Gamma$.
 The two ends of each of $\eta_i,\eta_j,\eta_k$ go into the cusp $O$ and the intersection of the  horocycle $\rho=\rho_0$ with the ideal triangle $\Delta_{ijk}$ consists of three components, denoted as $h^{ijk}_i, h^{ijk}_j,h^{ijk}_l$ (the red lines in Fig.~\ref{fig:triangle}). It is clear that the sum of the lengths $l_{h^{ijk}_i}(X)+l_{h^{ijk}_j}(X)+l_{h^{ijk}_k}(X)$ over all these $4g-3+n$ ideal triangles is less than the length of the horocycle $\rho=\rho_0$.
 Denote by $M^{ijk}_i$ the midpoint of $\eta_i$ with respect to the ideal triangle $\Delta_{ijk}$. $M_i$ together with $h^{ijk}_j,h^{ijk}_k$ cuts $\eta_i$ into four connected components, two of which are compact, denoted by $\eta^{ijk}_{i1}$ and $\eta^{ijk}_{i2}$ with respect to the counterclockwise order (see Fig.~\ref{fig:triangle}). It follows that $$l_{\eta^{ijk}_{i2}}(X)=l_{\eta^{ijk}_{j1}}(X);\ l_{\eta^{ijk}_{j2}}(X)=l_{\eta^{ijk}_{k1}}(X);\
 l_{\eta^{ijk}_{k2}}(X)=l_{\eta^{ijk}_{i1}}(X).$$
 Hence
 \begin{eqnarray}\label{eq:relation}
 {l_{\eta_i}(X)-l_{\eta_j}(X)}+l_{\eta_k}(X)&=&
 2l_{\eta^{ijk}_{i1}}(X)=2l_{\eta^{ijk}_{k2}}(X), \nonumber\\
 {l_{\eta_j}(X)-l_{\eta_k}(X)}+l_{\eta_i}(X)&=&
 2l_{\eta^{ijk}_{j1}}(X)=2l_{\eta^{ijk}_{i2}}(X),\\
 {l_{\eta_k}(X)-l_{\eta_i}(X)}+l_{\eta_j}(X)&=&
 2l_{\eta^{ijk}_{k1}}(X)=2l_{\eta^{ijk}_{j2}}(X)\nonumber.
 \end{eqnarray}
 Now ~(\ref{eq:def}) can be rewritten as
 \begin{eqnarray}\label{eq-hilb-teich-2}
   d^{\Gamma}_h(X_1,X_2)&=&\frac{1}{2}\sup_{\Delta_{ijk}\in \mathbb T_\Gamma} \max\{\log\frac{l_{\eta^{ijk}_{i1}}(X_1)}{l_{\eta^{ijk}_{i1}}(X_2)}
   ,\log\frac{l_{\eta^{ijk}_{j1}}(X_1)}{l_{\eta^{ijk}_{j1}}(X_2)}
   ,\log\frac{l_{\eta^{ijk}_{k1}}(X_1)}{l_{\eta^{ijk}_{k1}}(X_2)}\}\\
   && + \frac{1}{2}\sup_{\Delta_{ijk}\in \mathbb T_\Gamma  } \max\{\log\frac{l_{\eta^{ijk}_{i1}}(X_2)}{l_{\eta^{ijk}_{i1}}(X_1)}
   ,\log\frac{l_{\eta^{ijk}_{j1}}(X_2)}{l_{\eta^{ijk}_{j1}}(X_1)}
   ,\log\frac{l_{\eta^{ijk}_{k1}}(X_2)}{l_{\eta^{ijk}_{k1}}(X_1)}\}. \nonumber
 \end{eqnarray}
 And the length of the horocycle $h^{ijk}_i$ can be expressed in terms of $l_{\eta^{ijk}_{j2}}$ and $l_{\eta^{ijk}_{k1}}$ as $l_{h^{ijk}_i}(X)=\exp (-l_{\eta^{ijk}_{k1}}(X))=\exp (-l_{\eta^{ijk}_{j2}}(X))$. Therefore
 \begin{equation}\label{eq:rho}
   l_{\eta^{ijk}_{i1}}\geq \rho_0,\ l_{\eta^{ijk}_{j1}}\geq \rho_0,\ l_{\eta^{ijk}_{k1}}\geq \rho_0,\text{ for any ideal triangle } \Delta_{ijk}\in \mathbb T_\Gamma.
 \end{equation}

We summarize our discussions above as the following proposition.
\begin{proposition}\label{prop:Hilbert}
  With the notations  above, the Hilbert metric $d^\Gamma_h$ on the Teichm\"uller space $T_{g,n}$ can be expressed as the following two forms.
  \begin{itemize}
    \item  For any $X_1,X_2\in T_{g,n}$,
     \begin{eqnarray}\label{eq:def2}
     {d}^{\Gamma}_h(X_1,X_2)&=&\frac{1}{2}(\sup_{\Delta_{ijk}\in \mathbb T_\Gamma}\log\frac{{l_{\eta_i}(X_1)-l_{\eta_j}(X_1)+l_{\eta_k}(X_1)}}
    {{l_{\eta_i}(X_2)-l_{\eta_j}(X_2)+l_{\eta_k}(X_2)}}\\
      &&
     +\sup_{\Delta_{ijk}\in \mathbb T_\Gamma}\log\frac{{l_{\eta_i}(X_2)
     -l_{\eta_j}(X_2)+l_{\eta_k}(X_2)}}{{l_{\eta_i}
      (X_1)-l_{\eta_j}(X_1)+l_{\eta_k}(X_1)}})\nonumber,
      \end{eqnarray}
      where, ${{l_{\eta_i}
      (X)-l_{\eta_j}(X)+l_{\eta_k}(X)}}\geq 2\rho_0$ for any $X\in T_{g,n}$ and any ideal triangle $\Delta_{ijk}\in \mathbb T_\Gamma$.
     \item For any $X_1,X_2\in T_{g,n}$,
       \begin{eqnarray*}
       d^{\Gamma}_h(X_1,X_2)&=&\frac{1}{2}\sup_{\Delta_{ijk}\in \mathbb T_\Gamma}\max\{\log\frac{l_{\eta^{ijk}_{i1}}(X_1)}{l_{\eta^{ijk}_{i1}}(X_2)}
      ,\log\frac{l_{\eta^{ijk}_{j1}}(X_1)}{l_{\eta^{ijk}_{j1}}(X_2)}
      ,\log\frac{l_{\eta^{ijk}_{k1}}(X_1)}{l_{\eta^{ijk}_{k1}}(X_2)}\}\\
      && + \frac{1}{2}\sup_{\Delta_{ijk}\in \mathbb T_\Gamma}\max\{\log\frac{l_{\eta^{ijk}_{i1}}(X_2)}{l_{\eta^{ijk}_{i1}}(X_1)}
     ,\log\frac{l_{\eta^{ijk}_{j1}}(X_2)}{l_{\eta^{ijk}_{j1}}(X_1)}
     ,\log\frac{l_{\eta^{ijk}_{k1}}(X_2)}{l_{\eta^{ijk}_{k1}}(X_1)}\},
      \end{eqnarray*}
      where
       $
        l_{\eta^{ijk}_{i1}}\geq \rho_0,\ l_{\eta^{ijk}_{j1}}\geq \rho_0,\ l_{\eta^{ijk}_{k1}}\geq \rho_0
       $
       for any ideal triangle $\Delta_{ijk}\in \mathbb T_\Gamma$.
    \end{itemize}
\end{proposition}


\remark
The reason that we do not consider the Funk metric is that the Funk metric is not projectively invariant while the Hilbert metric is. If we  choose a fixed homogenous coordinate for $\mathbb RP^{6g-6+2n}$, we can define a Funk metric on $T_{g,n}$.

\vskip 20pt
\section{Earthquake}\label{sec:earthquake}
\vskip 10pt
\subsection{Earthquake}\label{sec:def-earth}
For any $X\in T_{g,n}$, $\beta\in \mathcal S$, denote by $l_\beta(X)$ the length of the geodesic representative of $\beta$ on $X$.
In \cite{Ke}, Kerckhoff proved that for any simple closed curve $\beta$, $l_{\beta}(\mathcal E^t_\alpha X)$ is a convex function of $t$ along the earthquake line $\{\mathcal E^t_\alpha X\}_{t\in \mathbb R}$. Based on this result, Bonahon proved  \cite{Bon} that each (anti-) earthquake ray converges to a unique point in the Thurston boundary $\partial T_{g,n}$.

\begin{lemma}[\cite{Ke}] \label{lem:Ke}
For any $X\in T_{g,n}$, $\alpha\in \mathcal{ML}$ and $\beta\in \mathcal S$, with $i(\alpha,\beta)>0$, then
  $$
  \frac{d }{dt}l_{\beta}(\mathcal E^t_\alpha X)=\int_{\beta} \cos\theta_t d\alpha,
  $$
  where $\theta_t$ represents the angle at each intersection point between (the geodesic representatives of) $\beta$ and $\alpha$ on $\mathcal E^t_\alpha X$ measured counter-clockwise from $\beta$  to $\alpha$. Moreover, as t tends to $+\infty$ (resp. $-\infty$), the function $t\mapsto \cos \theta_t$ is strictly increasing (resp. decreasing).
\end{lemma}

\begin{lemma}[\cite{Bon}]\label{lem:Bon}
  For every $X\in T_{g,n}$ and every $\alpha\in \mathcal {ML}$,
  $$\lim_{t\to \pm \infty}\frac{1}{|t|}i(\mathcal E^t_\alpha X,\beta)=i(\alpha,\beta),\text{ for any }\beta \in \mathcal S. $$
\end{lemma}

By applying similar arguments as in \cite {Ke} and \cite{Bon}, we get  similar results for $l_{\eta_i}(\mathcal E^t_\alpha X)$ (note that $l_{\eta_i}(\mathcal E^t_\alpha X)$ is the length of a particular compact subarc of $\eta_i$).

\begin{lemma} \label{lem:derivative}
 For any $X\in T_{g,n}$, $\alpha\in \mathcal{ML}$ and $\eta_i\in \Gamma$,
  $$
  \frac{d }{dt}l_{\eta_i}(\mathcal E^t_\alpha X)=\int_{\eta_i} \cos\theta_t d\alpha,
  $$
  where $\theta_t$ represents the angle at each intersection point between (the geodesic representatives of) $\eta_i$ and $\alpha$ on $\mathcal E^t_\alpha X$ measured counter-clockwise from $\eta_i$  to $\alpha$. Moreover, as t tends to $+\infty$ (resp. $-\infty$),
  the function $t\mapsto \cos \theta_t$ is strictly increasing (resp. decreasing).

\end{lemma}
\begin{proof}
  The proof is exactly the same as that of Lemma~\ref{lem:Ke} in \cite{Ke}.
\end{proof}
\begin{lemma}\label{lem:eta-convergence}
  For every $X\in T_{g,n}$ and every $\alpha\in \mathcal {ML}$,
  $$\lim_{t\to \pm \infty}\frac{1}{|t|}l_{\eta_i}(\mathcal E^t_\alpha X)=i(\alpha,\eta_i),\text{ for any }\eta_i\in \Gamma. $$
\end{lemma}
\begin{proof}
  The proof is exactly the same as that of Lemma~\ref{lem:Bon} in \cite{Bon}. Also, it follows from Lemma~\ref{lem:estimate}.
\end{proof}
In fact, we can prove a stronger result.
\begin{lemma}\label{lem:estimate}
  For every $X\in T_{g,n}$ and every $\alpha\in \mathcal {ML}$,
  $$|l_{\eta_i}(\mathcal E^t_\alpha X)-|t| i(\alpha, \eta_i)|\leq l_{\eta_i}(X)+e^{-\rho_0},\text{ for any } \eta_i\in \Gamma\text{ and } t\in \mathbb R,$$
  where $e^{-\rho_0}$ represents the length of the horocycle centered at the puncture $O$  with height $\rho_0$. Moreover, there is a constant $c^+_i,\ c^-_i\in [-l_{\eta_i}(X)-e^{-\rho_0},+l_{\eta_i}(X)+e^{-\rho_0}]$ such that
  $$\lim_{t\to +\infty} [l_{\eta_i}(\mathcal E^t_\alpha X)-t i(\alpha, \eta_i)]=c^+_i,\ \lim_{t\to -\infty} [l_{\eta_i}(\mathcal E^t_\alpha X)-|t| i(\alpha, \eta_i)]=c^-_i.$$
\end{lemma}
\begin{proof}
  By symmetry, it suffices to prove the lemma for $t\geq 0$. We start with the case that $\alpha$ is a simple closed geodesic.
  Let $f_i(t)\triangleq l_{\eta_i}(\mathcal E^t_\alpha X)-ti(\eta_i,\alpha)$. We will show that,
   \begin{equation}\label{eq:simple}
   |f_i(t)|\leq l_{\eta_i}(X)+e^{-\rho_0},\text{ for any } \eta_i\in \Gamma\text{ and } t\geq 0.
   \end{equation}

   From Lemma \ref{lem:derivative},
  $$ \frac{df_i(t)}{dt}<0,\text{ for any } t\geq 0.$$

  As a consequence, we get our first inequality immediately,
  $$f_i(t)\leq f_i(0)=l_{\eta_i}(X).$$

  By the monotonicity of $f_i(t)$, for any positive integer $m$, we have
  $$  f_i(t)\geq f_i(ml_\alpha(X)),
  \text{ for any } t\in[0,ml_\alpha(X)]. $$

  Therefore, to prove $f_i(t)\geq -l_{\eta_i}(X)-e^{-\rho_0}$ for all $t\geq 0$, it suffices to prove  $f_i(ml_\alpha(X))\geq -l_{\eta_i}(X)-e^{-\rho_0}$ for all $m\in \mathbb N^+$.

  \begin{figure}
    \includegraphics[width=120mm]{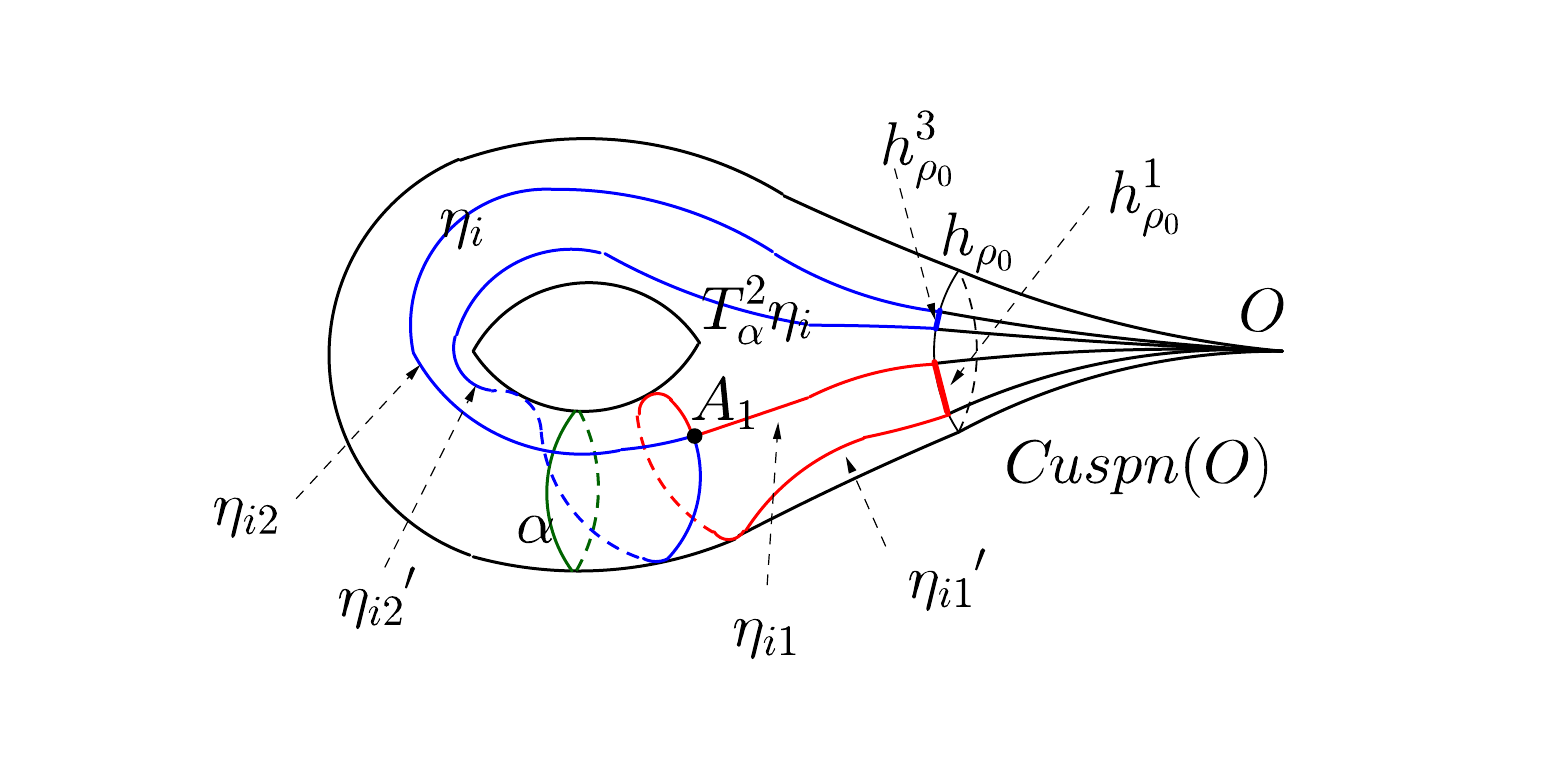}

    \caption{Dehn twist. $T^2_\alpha \beta$ is obtained from $\beta$ by two times Dehn twist along $\alpha$. }
    \label{fig:closed-curves}
  \end{figure}
  Recall that $\mathcal E^{ml_\alpha(X)}_\alpha X$ is obtained from $X$ by $m$ times Dehn twist along $\alpha$. Denote by $T^m_\alpha$ the $m$ times Dehn twist along $\alpha$. Hence $l_{\eta_i}(\mathcal E^{ml_\alpha(X)}_\alpha X)=l_{T^{-m}_\alpha\eta_i}(X)$. Denote by $h_{\rho_0}$ the horocycle centered at the puncture  $O$ with height $\rho_0$, and denote by $\eta_i^*$, $(T^{-m}_\alpha \eta_i)^*$ the geodesic representatives of $\eta_i$, $T^{-m}_\alpha \eta_i$ on $X$, respectively. It is clear that $\eta_i^*$ and $(T^{-m}_\alpha \eta_i)^*$ have $mi(\eta_i,\alpha)$ intersection points, one of which  is the puncture $O$. The remaining $mi(\eta_i,\alpha)-1$ intersection points cut the restrictions of  $\eta_i^*$ and $(T^{-m}_\alpha \eta_i)^*$ on $X\backslash \textup{Cuspn}(O)$ into $2mi(\eta_i,\alpha)$ segments $\eta_{i1},\eta_{i1}',...,\eta_{il},\eta_{il}'$, where $l=mi(\eta_i,\alpha)$ and  $\textup{Cuspn}(O)$ represents the cusp neighbourhood of $O$ with boundary horocycle $h_{\rho_0}$. On the other hand, $\eta_i^*$ and $(T^{-m}_\alpha \eta_i)^*$ cut the horocycle $h_{\rho_0}$ into four segments $h^1_{\rho_0},h^2_{\rho_0},h^3_{\rho_0},h^4_{\rho_0}.$ Two of these four segments, say $h^1_{\rho_0},h^3_{\rho_0}$, together with above mentioned $2mi(\eta_i,\alpha)$ segments $\eta_{i1},\eta_{i1}',...,\eta_{il},\eta_{il}'$ consist $mi(\eta_i,\alpha)$ simple closed curves, and all of them are homotopic to $\alpha$ (see Fig.~\ref{fig:closed-curves}). More precisely, let
  $\gamma_1=\eta_{i1}*\eta_{i1}'* h_{\rho_0}^1$,
  $\gamma_2=\eta_{i2}*\eta_{i2}'$, $\cdots$,
  $\gamma_{l-1}=\eta_{i,l-1}*\eta_{i,l-1}'$,
  $\gamma_l=\eta_{il}*\eta_{il}'* h_{\rho_0}^3$. $\gamma_1,\cdots,\gamma_l$ are simple closed curves, and all of them are isotopic to $\alpha$.
   It follows that
  $$mi(\eta_i,\alpha)l_\alpha(X)\leq l_{T^{-m}_\alpha \eta_i}(X)+l_{\eta_i}(X)+e^{-\rho_0},\text{ for any } m\in \mathbb N^+,$$
  which means that
  $$
  f_i(ml_\alpha(X))\geq -l_{\eta_i}(X)-e^{-\rho_0},\text{ for any } m\in \mathbb N^+.
  $$
  This completes the proof of (\ref{eq:simple}).
  \vskip 5pt
  Now we extend these estimates to an arbitrary measured geodesic lamination $\mu$. Recall that $\mathbb R^+\times \mathcal S$ is dense in $\mathcal {ML}$. There exists $(s_m,\alpha_m)\in\mathbb R^+\times \mathcal S$ such that $s_m\alpha_m\to \mu$, as $m\to \infty$. It follows from Proposition \ref{prop:earthquake} that for any given $t\in \mathbb R$, $\mathcal E^t_{s_m\alpha_m}X\to\mathcal E^t_\mu X$ as $m\to \infty$. Hence for any given $t\in\mathbb R$, $l_{\eta_i}(\mathcal E^t_{s_m\alpha_m}X)\to l_{\eta_i}(\mathcal E^t_\mu X)$ as $m\to \infty$. From ~(\ref{eq:simple}), we have
  $$ |ts_mi(\alpha_m,\eta_i)-l_{\eta_i}(\mathcal E^{t}_{s_m\alpha_m}X)|\leq l_{\eta_i}(X)+e^{-\rho_0}\ \text{for any } m\in \mathbb N^+,\ \eta_i\in\Gamma\text{ and } t\geq 0.$$
  Let $m$ tend to infinity, we have
  $$|l_{\eta_i}(\mathcal E^t_\mu X)-t i(\mu, \eta_i)|\leq l_{\eta_i}(X)+e^{-\rho_0}, \text{ for any } \eta_i\in \Gamma\text{ and } t\geq0.$$

  \vskip 5pt
  The existence of the limit $\lim_{t\to +\infty}[l_{\eta_i}(\mathcal E^t_\mu X)-t i(\mu, \eta_i)]$ follows from the boundedness and the monotonicity of $l_{\eta_i}(\mathcal E^t_\mu X)-t i(\mu, \eta_i)$.
\end{proof}
\vskip 10pt

Based on the estimates in Lemma~\ref{lem:estimate}, we can describe the coordinates of an earthquake line $\{\mathcal E^t_\mu X\}_{t\in \mathbb R}$. Denote by $P_1$ the hyperplane
 $$\{(x_1,x_2,...,x_{6g-5+2n}):x_1+x_2+...+x_{6g-5+2n}=1\}, $$
 and  denote by $\pi$ the projective map
 $$
 \begin{array}{cccc}
  \pi:&\mathbb R_+^{6g-5+2n} &\longrightarrow & P_1\\
   &(x_1,x_2,...,x_{6g-5+2n})&\longmapsto &\frac{(x_1,x_2,...,x_{6g-5+2n})}{x_1+x_2+...+x_{6g-5+2n}}.
  \end{array}
 $$
 For a measured geodesic lamination $\mu$, let $I(\mu)=\Sigma_{i=1}^{6g-5+2n}i(\eta_i,\mu)$.  It is clear that $\pi\circ \Lambda(\mu)=I(\mu)^{-1}(i(\eta_1,\mu),...,i(\eta_{6g-5+2n},\mu)).$
 Moreover, it follows from Lemma~\ref{lem:estimate} that for large enough $|t|$,
 $$\pi\circ\Lambda(\mathcal E^t_\mu X)=\pi\circ \Lambda(\mu)+\frac{1}{|t|}\frac{c_i}{I(\mu)}+o(\frac{1}{|t|}). $$
 In other words, the images of the earthquake rays $\{\mathcal E^t_\mu X\}_{t\geq T}$ and $\{\mathcal E^t_\mu X\}_{t\leq -T}$ in $P_1$ look like straight line segments when $T$ is large enough.

\vskip 10pt
 Now we  prove the main theorem of this paper.
\begin{mtheorem}
  After reparametrization, every (anti-)earthquake ray is an almost-geodesic in $(T_{g,n},d^{\Gamma}_h)$.
\end{mtheorem}
\begin{proof}
  By symmetry, it suffices to prove the theorem for every earthquake ray.
  Let $\{\mathcal E^{t}_{\mu}(X)\}_{t\geq0}$ be an earthquake ray directed by $\mu$ and starting at $X$. Set $X_t=\mathcal E^{t}_{\mu}(X)$.
  Let $f_i(t)=l_{\eta_i}(\mathcal E^t_\mu X)-t i(\mu,\eta_i)$. It follows from Lemma \ref{lem:estimate} that
 $$|f_i(t)|<l_{\eta_i}(X)+e^{-\rho_0},$$
  and that for any $\epsilon>0$ there is a constant $T_1>0$ depending on $\epsilon$, $X$, $\Gamma$ and $\mu$ such that
  $$ |f_i(t)-c_i|<\epsilon,\text{ for any } t>T_1. $$
 Moreover, if $i(\mu,\eta_j)=0$ for some $j$,  $l_{\eta_j}(\mathcal E^t_\mu X)\equiv l_{\eta_j}(X)$.

 \vskip 5pt
 Next we divide the set of ideal triangles corresponding to the preferred triangulation $\Gamma$ into three types:
  \begin{itemize}
    \item \Type A consists of  ideal triangles $\Delta_{ijk}$ which do not intersect $\mu$, i.e. $i(\eta_i,\mu)=i(\eta_j,\mu)=i(\eta_k,\mu)=0$;
    \item \Type B consists of  ideal triangles $\Delta_{ijk}$ whose two sides , say $\eta_i,\eta_j$,   intersect $\alpha$ and whose remaining side does not intersect $\mu$, i.e. $i(\eta_i,\mu)=i(\eta_j,\mu)>0$ and $i(\eta_k,\mu)=0$;
    \item \Type C consists of ideal triangles $\Delta_{ijk}$ with each side  intersecting $\mu$,  i.e. $i(\eta_i,\mu)>0,i(\eta_j,\mu)>0,i(\eta_k,\mu)>0$.
  \end{itemize}
 For each  triangle $\Delta_{ijk}$, set
  \begin{equation}\label{eq:d-ijk}
  d_{ijk}=\frac{c_i-c_j+c_k}
 {{l_{\eta_i}(X_0)-l_{\eta_j}(X_0)+l_{\eta_k}(X_0)}}\textup{  and }
 \bar d_{ijk}=\frac{i(\eta_i,\mu)-i(\eta_j,\mu)+i(\eta_k,\mu)}
 {{l_{\eta_i}(X_0)-l_{\eta_j}(X_0)+l_{\eta_k}(X_0)}}.
 \end{equation}
 \vskip 10pt

 Now we discuss case by case.
 \vskip 5pt
  \Type A. In this case, $i(\eta_i,\mu)=i(\eta_j,\mu)=i(\eta_k,\mu)=0$.
 $$ \frac{{l_{\eta_i}(X_s)-l_{\eta_j}(X_s)+l_{\eta_k}(X_s)}}
{{l_{\eta_i}(X_t)-l_{\eta_j}(X_t)+l_{\eta_k}(X_t)}}
 \equiv \frac{{l_{\eta_i}(X_0)-l_{\eta_j}(X_0)+l_{\eta_k}(X_0)}}
{{l_{\eta_i}(X_0)-l_{\eta_j}(X_0)+l_{\eta_k}(X_0)}}
\equiv 1,\text{ for any } s,t\geq0.
 $$

 \vskip 10pt
  \Type B. In this case, $i(\eta_i,\mu)=i(\eta_j,\mu)>0$ and $i(\eta_k,\mu)=0$.
  Recall that for any ideal triangle $\Delta_{i'j'k'}$ (see Proposition~\ref{prop:Hilbert}),
 $$ \frac{{l_{\eta_{i'}}(X_s)-l_{\eta_{j'}}(X_s)+l_{\eta_{k'}}(X_s)}}{2} \geq \rho_0>0. $$
 Hence for any $\epsilon>0$, there is a constant $T_2>0$ depending on $\epsilon$, $\Gamma$, $X$ and $\mu$ such that \text{ for any } $s,t>T_2$, we have
 \begin{eqnarray*}
  \frac{{l_{\eta_i}(X_s)-l_{\eta_j}(X_s)+l_{\eta_k}(X_s)}}
 {{l_{\eta_i}(X_t)-l_{\eta_j}(X_t)+l_{\eta_k}(X_t)}}
 &=&\frac{f_i(s)-f_j(s)+l_{\eta_k}(X)}{f_i(t)-f_j(t)+l_{\eta_k}(X)}\\
 &\in& (1-\epsilon,1+\epsilon),\\
  \frac{{l_{\eta_k}(X_s)-l_{\eta_i}(X_s)+l_{\eta_j}(X_s)}}
{{l_{\eta_k}(X_t)-l_{\eta_i}(X_t)+l_{\eta_j}(X_t)}}&\in & (1-\epsilon,1+\epsilon),
\end{eqnarray*}
and
\begin{eqnarray*}
   \frac{{l_{\eta_j}(X_s)-l_{\eta_k}(X_s)+l_{\eta_i}(X_s)}}
   {{l_{\eta_j}(X_t)-l_{\eta_k}(X_t)+l_{\eta_i}(X_t)}}
   &=&
   \frac{i(\eta_j,\mu)s+i(\eta_i,\mu)s+f_i(s)+f_j(s)-l_{\eta_k}(X)}
   {{i(\eta_j,\mu)t+i(\eta_i,\mu)t+f_i(s)+f_j(s)-l_{\eta_k}(X)}}\\
   &\in& ((1-\epsilon)\frac{s}{t},(1+\epsilon)\frac{s}{t}).
 \end{eqnarray*}

 Moreover, there is $T_2'>0$  depending on $\epsilon$, $\Gamma$, $X$ and $\mu$ such that \text{ for any } $s,t>T_2'$, we have
 \begin{eqnarray*}
  \frac{{l_{\eta_i}(X_s)-l_{\eta_j}(X_s)+l_{\eta_k}(X_s)}}
 {{l_{\eta_i}(X_0)-l_{\eta_j}(X_0)+l_{\eta_k}(X_0)}}
 &=&\frac{f_i(s)-f_j(s)+l_{\eta_k}(X)}
 {{l_{\eta_i}(X_0)-l_{\eta_j}(X_0)+l_{\eta_k}(X_0)}}\\
 &\in& ((1-\epsilon)d_{ijk},(1+\epsilon)d_{ijk}),\\
  \frac{{l_{\eta_k}(X_s)-l_{\eta_i}(X_s)+l_{\eta_j}(X_s)}}
 {{l_{\eta_k}(X_0)-l_{\eta_i}(X_0)+l_{\eta_j}(X_0)}}
 &\in& ((1-\epsilon)d_{kij},(1+\epsilon)d_{kij}),
 \end{eqnarray*}
 and
 \begin{eqnarray*}
 \frac{{l_{\eta_j}(X_s)-l_{\eta_k}(X_s)+l_{\eta_i}(X_s)}}
 {{l_{\eta_j}(X_0)-l_{\eta_k}(X_0)+l_{\eta_i}(X_0)}}
 &\in& ((1-\epsilon)s\bar d_{jki},(1+\epsilon)s\bar d_{jki}),
 \end{eqnarray*}
 where $d_{ijk}$ and $\bar d_{ijk}$ are defined by ~(\ref{eq:d-ijk}).
 \vskip 10pt
 \Type C. In this case, $i(\eta_i,\mu)>0,i(\eta_j,\mu)>0,i(\eta_k,\mu)>0$.
 We distinguish two subcases.

 \vskip 5pt
  If $i(\eta_i,\mu)-i(\eta_j,\mu)+i(\eta_k,\mu)>0$, it follow from Lemma~
  \ref{lem:estimate} that there is a constant $T_3>0$ depending on $\epsilon$, $\Gamma$, $X$ and $\mu$ such that \text{ for any } $s,t>T_3$, we have
 \begin{eqnarray*}
   && \frac{{l_{\eta_i}(X_s)-l_{\eta_j}(X_s)+l_{\eta_k}(X_s)}}
    {{l_{\eta_i}(X_t)-l_{\eta_j}(X_t)+l_{\eta_k}(X_t)}}\\
    &=&
    \frac{i(\eta_i,\mu)s-i(\eta_j,\mu)s+i(\eta_k,\mu)s+f_i(s)-f_j(s)+f_k(s)}
    {i(\eta_i,\mu)t-i(\eta_j,\mu)t+i(\eta_k,\mu)t+f_i(s)-f_j(s)+f_k(s)}\\
    &\in& ((1-\epsilon)\frac{s}{t},(1+\epsilon)\frac{s}{t}),\\
    \textup{and}\\
       && \frac{{l_{\eta_i}(X_s)-l_{\eta_j}(X_s)+l_{\eta_k}(X_s)}}
    {{l_{\eta_i}(X_0)-l_{\eta_j}(X_0)+l_{\eta_k}(X_0)}}\\
   &\in& ((1-\epsilon)s\bar d_{ijk},(1+\epsilon)s\bar d_{ijk}).
 \end{eqnarray*}

 If $i(\eta_i,\mu)-i(\eta_j,\mu)+i(\eta_k,\mu)=0$, there is a constant $T_3'>0$ depending on $\epsilon$, $\Gamma$, $X$ and $\mu$ such that \text{ for any } $s,t>T_3'$, we have
 \begin{eqnarray*}
    &&\frac{{l_{\eta_i}(X_s)-l_{\eta_j}(X_s)+l_{\eta_k}(X_s)}}
    {{l_{\eta_i}(X_t)-l_{\eta_j}(X_t)+l_{\eta_k}(X_t)}}\\
    &=&\frac{i(\eta_i,\mu)s-i(\eta_i,\mu)s+i(\eta_k,\mu)s+f_i(s)-f_j(s)+f_k(s)}
    {i(\eta_i,\mu)t-i(\eta_i,\mu)t+i(\eta_k,\mu)t+f_i(s)-f_j(s)+f_k(s)}\\
    &=&\frac{f_i(s)-f_j(s)+f_k(s)}
    {f_i(t)-f_j(t)+f_k(t)}\\
    &\in& (1-\epsilon,1+\epsilon),\\
    \textup{and}\\
     &&\frac{{l_{\eta_i}(X_s)-l_{\eta_j}(X_s)+l_{\eta_k}(X_s)}}
    {{l_{\eta_i}(X_0)-l_{\eta_j}(X_0)+l_{\eta_k}(X_0)}}\\
    &\in& ((1-\epsilon)d_{ijk},(1+\epsilon)d_{ijk}).
 \end{eqnarray*}

 \vskip 5pt
 Therefore, for any $t\geq s\geq \max\{T_1,T_2,T_2',T_3,T_3'\}$,
 \begin{eqnarray}\label{eq:dist-st}
 \nonumber    {d}^{\Gamma}_h(X_s,X_t)&=&\frac{1}{2}(\sup_{\Delta_{ijk}\in \mathbb T_\Gamma}\log\frac{{l_{\eta_i}(X_t)-l_{\eta_j}(X_t)+l_{\eta_k}(X_t)}}
  {{l_{\eta_i}(X_s)-l_{\eta_j}(X_s)+l_{\eta_k}(X_s)}}\\
 \nonumber &&
  +\sup_{\Delta_{ijk}\in \mathbb T_\Gamma}\log\frac{{l_{\eta_i}(X_s)
  -l_{\eta_j}(X_s)+l_{\eta_k}(X_s)}}{{l_{\eta_i}
   (X_t)-l_{\eta_j}(X_t)+l_{\eta_k}(X_t)}})\\
  &\in& (\frac{1}{2}\log(1-\epsilon)\frac{t}{s},
  \frac{1}{2}\log(1+\epsilon)\frac{t}{s}+\frac{1}{2}\log(1+\epsilon)),
 \end{eqnarray}
 and
 \begin{eqnarray}\label{eq:dist-0s}
 \nonumber  {d}^{\Gamma}_h(X_s,X_0)&=&\frac{1}{2}(\sup_{\Delta_{ijk}\in \mathbb T_\Gamma}\log\frac{{l_{\eta_i}(X_s)-l_{\eta_j}(X_s)+l_{\eta_k}(X_s)}}
  {{l_{\eta_i}(X_0)-l_{\eta_j}(X_0)+l_{\eta_k}(X_0)}}\\
 \nonumber &&+\sup_{\Delta_{ijk}\in \mathbb T_\Gamma}\log\frac{{l_{\eta_i}(X_0)
  -l_{\eta_j}(X_0)+l_{\eta_k}(X_0)}}{{l_{\eta_i}
   (X_s)-l_{\eta_j}(X_s)+l_{\eta_k}(X_s)}})\\
  &\in&(\frac{1}{2}\log\frac{(1-\epsilon)s\bar  d}{(1+\epsilon)d},
  \frac{1}{2}\log\frac{(1+\epsilon)s\bar d}{(1-\epsilon)d}),
 \end{eqnarray}
 where $ d\triangleq \min\{d_{ijk}:i(\eta_i,\mu)-i(\eta_j,\mu)+i(\eta_k,\mu)=0,\ i(\eta_i,\mu)+i(\eta_j,\mu)+i(\eta_k,\mu)>0\}$, and $\bar d\triangleq \max\{\bar d_{ijk}:i(\eta_i,\mu)-i(\eta_j,\mu)+i(\eta_k,\mu)>0\}$.

 \vskip 5pt
 Now, reparametrize the earthquake ray $\{\mathcal E^t_\mu X\}_{t\geq 0}$ as $\{E^t_\mu X\}_{t\geq 0}$ by setting $E^t_\mu X=\mathcal E^{ (d/\bar d) \exp(2t)}_\mu X$. It follows from ~(\ref{eq:dist-st}) and ~(\ref{eq:dist-0s}) that for any $\epsilon>0$ there is a $T>0$ depending on $\epsilon$, $\Gamma$, $X$ and $\mu$ such that
 $$|d^\Gamma_h(X,E^s_\mu X)-s|< \epsilon,\text{ for any } s>T,$$
 and
 $$|d^\Gamma_h(E^t_\mu X,E^s_\mu X)-(t-s)|< \epsilon,\text{ for any }t\geq s>T.$$
 Therefore $\{E^t_\mu X\}_{t\geq 0}$ is an almost geodesic in $(T_{g,n},d^\Gamma_h)$.

\end{proof}

\vskip 10pt
\subsection{The horofunction boundary of $({T}_{g,n},d^\Gamma_{h})$}\label{sec:horo-boundary}

Let $(M,d)$ be a proper geodesic  metric space, which is endowed with the topology induced  by the metric $d$. We embed $(M,d)$ into $C(M)$, the space of continuous real-valued functions on $X$ endowed with the topology of uniform convergence on bounded sets,  by the map below:
\begin{eqnarray*}
  h:M&\longrightarrow &C(M)\\
  z&\longmapsto&[M\ni x \mapsto d(x,z)-d(b,z)],
\end{eqnarray*}
where $b\in M$ is a base point.
The {\it horofunction boundary} of $(M,d)$ is defined to be
$$
\partial{\overline{M_b}^{horo}}\triangleq \overline{h(M)}\backslash{h(M)},
$$
where $\overline{h(M)}$ represents the closure of $h(M)$ in $C(M)$.   The horofunction boundary is independent of the base point, i.e. $\partial{\overline{M_b}^{horo}}$ is homeomorphic to $\partial{\overline{M_{b'}}^{horo}}$ for $b,b'\in M$. A function in $\partial{\overline{M}^{horo}}$ is called a {\it horofunction}.

Rieffel (\cite{Rf}) observed   that every almost geodesic converges to a unique point in the horofunction boundary of $(X,d)$.  Therefore,  we have the following corollary.
\begin{corollary} \label{cor:convergence}
  Every (anti-) earthquake ray converges to a unique point in the horofunction boundary of $(T_{g,n},d^\Gamma_h)$.
\end{corollary}

\remark  As we mentioned in the beginning of this section,  Bonahon proved  \cite{Bon} that each (anti-) earthquake ray converges to a unique point in the Thurston boundary $\partial T_{g,n}$. It follows from Theorem \ref{thm:Teich} that each (anti-) earthquake ray converges to a point in the Euclidean boundary of the polytope $P$. In general, let  $D$ be a bounded convex domain  in the Euclidean space. Foertsch-Karlsson (\cite{FK}) proved that every geodesic under the Hilbert metric converges to a boundary point in $\partial D$ in the Euclidean sense. Walsh (\cite{Walsh2}) proved that  every sequence converging  to a point in the horofunction boundary of the Hilbert geometry converges to a point in the Euclidean boundary $\partial D$. Hence every almost geodesic  under the Hilbert metric converges to a boundary point in $\partial D$.

\section{Dependence of $d^{\Gamma}_{h}$ on the preferred triangulation $\Gamma$} \label{sec:triangulation}
In this section, we investigate the dependence of the Hilbert metric $d^{\Gamma}_h$ on the choice of triangulation $\Gamma$.  First of all, we define a basic operation for a triangulation, namely, \textit{diagonal-flip}.  


\begin{figure}
  \subfigure[]
  {
  \begin{minipage}[tbp]{50mm}
  \includegraphics[width=50mm]{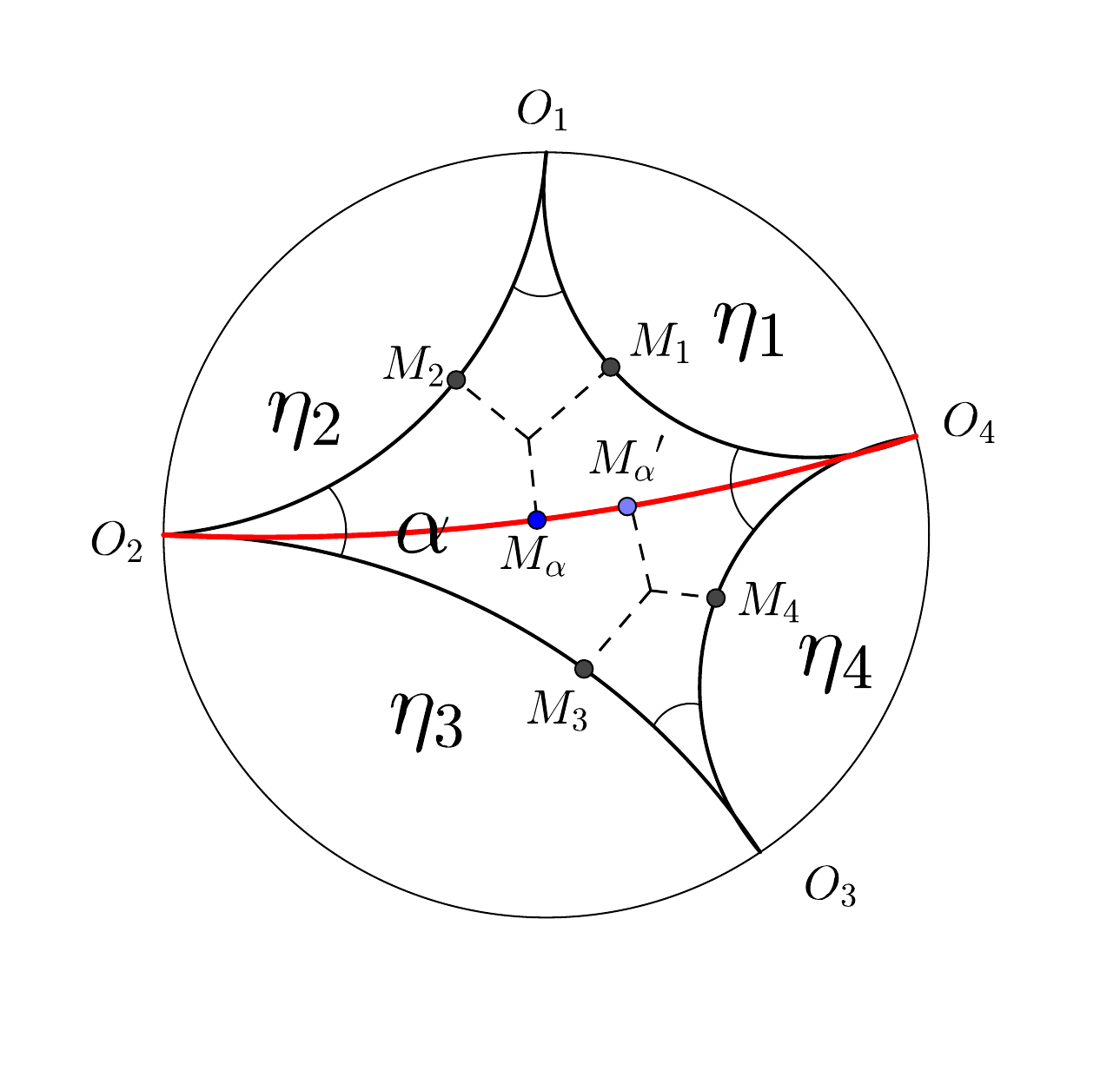}
  \label{fig:flip-1}
  \end{minipage}
  }
   \begin{minipage}[tbp]{15mm}
    \begin{tikzpicture}
      \draw [->](-0.5,0)--(0.5,0)
         node[pos=0.5,above] {diagonal-flip};
    \end{tikzpicture}
    \end{minipage}
  \subfigure[]
  {
  \begin{minipage}[tbp]{50mm}
  \includegraphics[width=50mm]{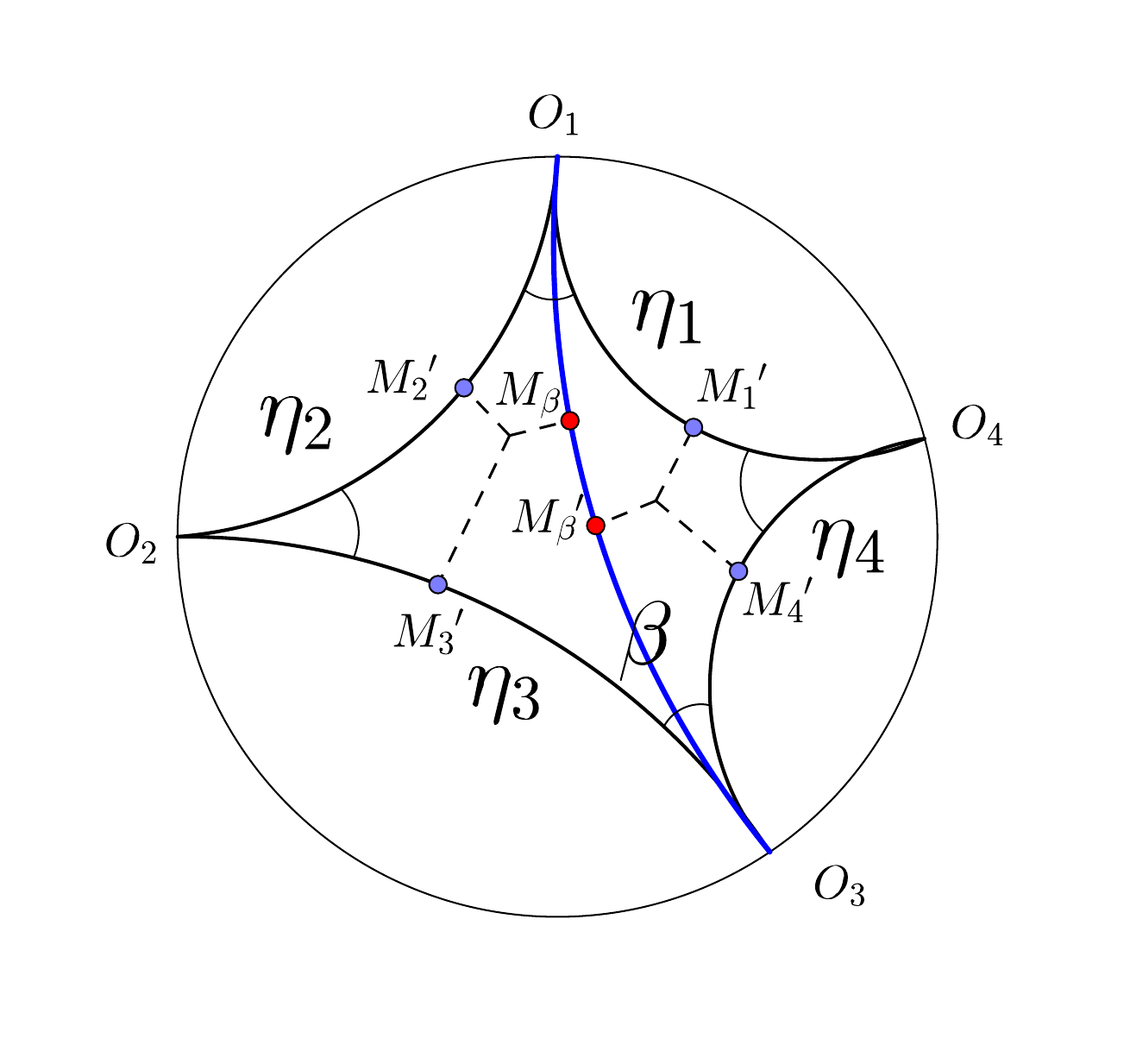}
  \label{fig:flip-2}
  \end{minipage}
  }

  \caption{Flip}
  \label{fig:flip}
\end{figure}

\begin{figure}

  \subfigure[]
  {
  \begin{minipage}[tbp]{50mm}
  \includegraphics[width=50mm]{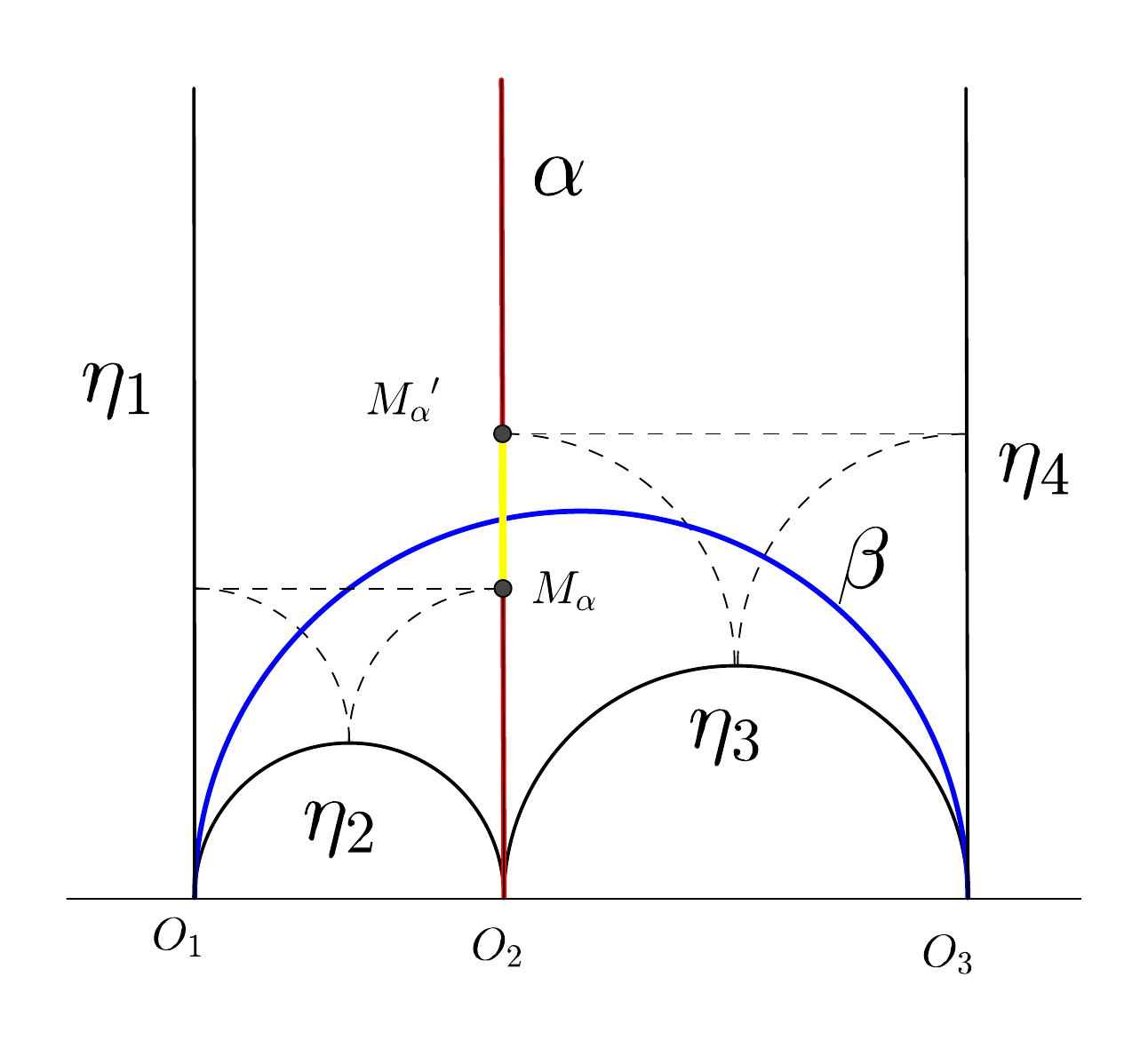}
  \label{fig:shearing-flip-1}
  \end{minipage}
  }
   \begin{minipage}[tbp]{15mm}
    \begin{tikzpicture}
      \draw [->](-0.5,0)--(0.5,0)
         node[pos=0.5,above] {$z\mapsto \frac{-z}{z-1-R}$ };
    \end{tikzpicture}
    \end{minipage}
  \subfigure[]
  {
  \begin{minipage}[tbp]{50mm}
  \includegraphics[width=50mm]{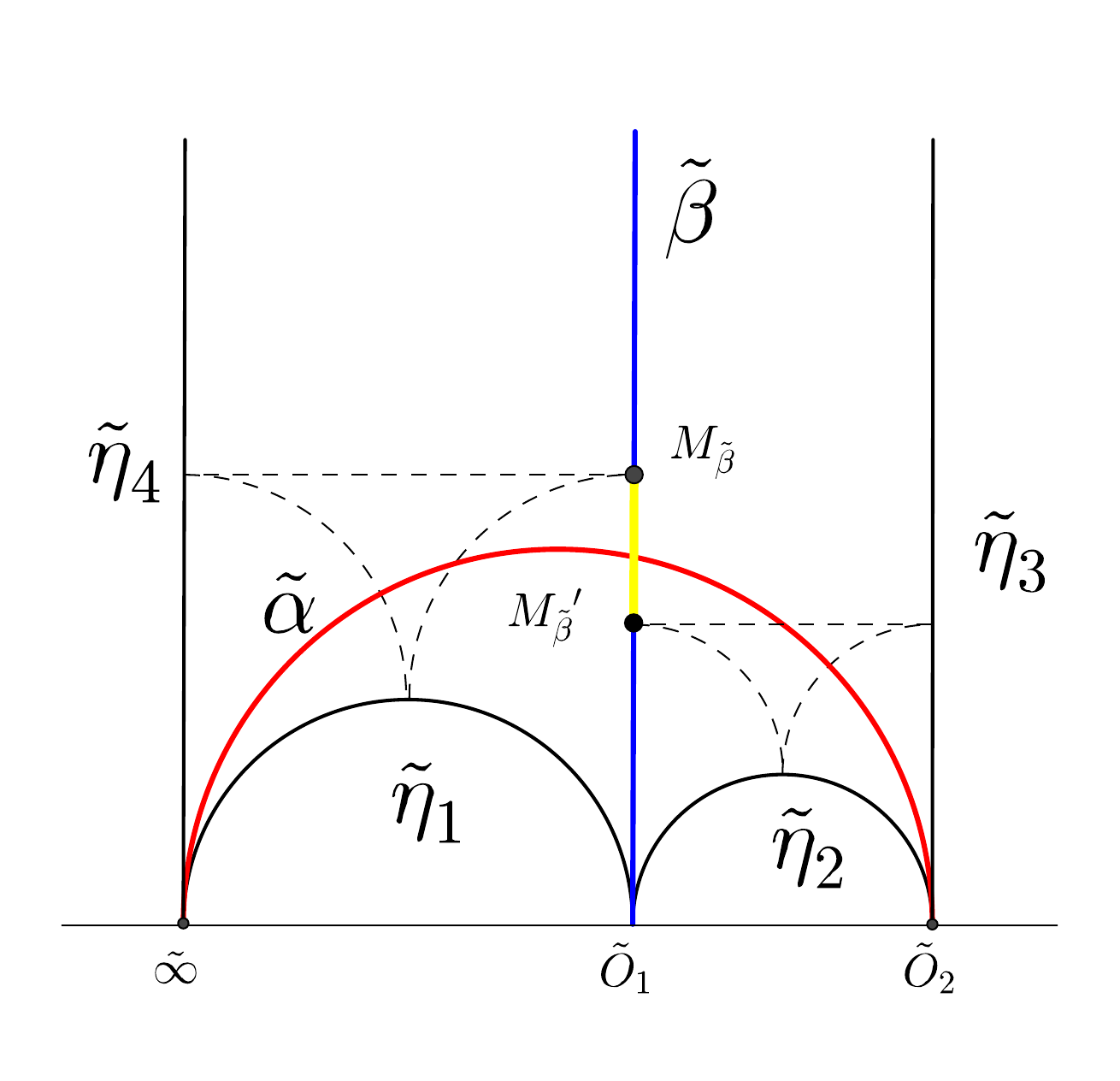}
  \label{fig:shearing-flip-2}
  \end{minipage}
  }

  \caption{Shearing}
  \label{fig:shearing}
\end{figure}
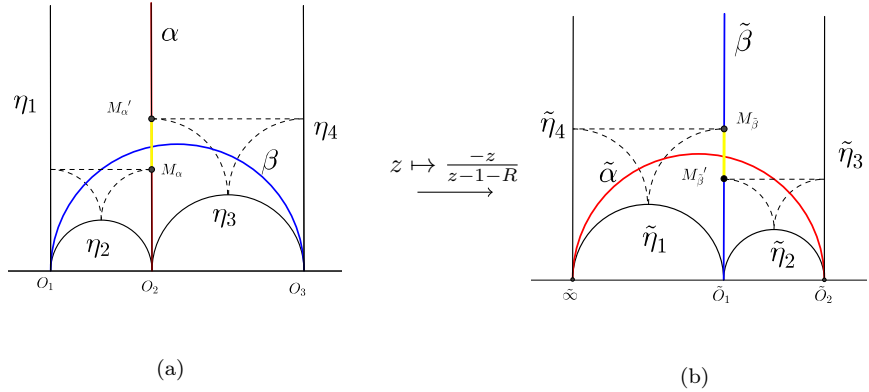

 Let $ Q$ be an ideal quadrilateral with  sides $\eta_1,\eta_2,\eta_3,\eta_4$ and  ideal vertices (corners) $O_1,O_2,O_3,O_4$. Let $H_i$ be a horocycle around the corner $O_i$ whose length is smaller than $e^{-\rho_0}$, $i=1,2,3,4$, where $\rho_0$ is chosen in Section \ref{sec:geometri para}. Let $\alpha$  be a diagonal geodesic connecting $O_2,O_4$ which triangulates $\mathcal Q$ into two triangles $\Delta_1$ and $\Delta_2$, and $M_1,M_2,M_3,M_4$  the ``midpoint'' of $\eta_1,\eta_2,\eta_3,\eta_4$ with respect to these two triangles. Further, let $M_{\alpha}$ and $M_{\alpha}'$ be the midpoints of $\alpha$  with respect to $\Delta_1$ and $\Delta_2$ respectively. Denote by $l_{i1}$ and $l_{i2}$ the distances from $M_i$  to $H_{i-1}$ and from $M_i$  to $H_{i}$ respectively.  Let $\beta$ be  the other diagonal geodesic, and the notations $M'_1,M'_2,M'_3,M'_4$, $M_{\beta},M_{\beta}'$, $l'_{i1}$ and $l'_{i2}$ are defined similarly. (see Fig.~\ref{fig:flip}.) 

 \begin{defi}
   With notations described above, we say that $\beta$ is obtained from $\alpha$ by a {\it diagonal-flip}  with respect to $Q$. A triangulation $\Gamma'$ is said to be obtained from $\Gamma$ by a diagonal-flip if $\Gamma'\backslash \{\alpha'\}=\Gamma\backslash\{\alpha\}$ and that $\alpha'$ can be obtained from $\alpha$ by a diagonal-flip with respect to some quadrilateral whose four sides are contained in $\Gamma'\backslash \{\alpha'\}=\Gamma\backslash\{\alpha\}$.
 \end{defi}

 \begin{defi}
   We define the \textit{shearing of $Q$ with respect to $\alpha$}, denoted by $\textup{shr}(\alpha,Q)$, in the following way. The absolute value of $\textup{shr}(\alpha,Q)$ is defined to be the distance between $M_\alpha$ and $M_\alpha'$, and the sign of $\textup{shr}(\alpha,Q)$ is defined to ``+'' if $M_\alpha'$ sits on the left side of $M_\alpha$ observed from $\Delta_1$, otherwise the sign is defined to be ``-''.
 \end{defi}
 \remark The ``left'' notation in the definition of the shearing depends only on the orientation of the ideal quadrilateral.

 The lemma below describes some basic properties of a diagonal-flip.
 \begin{lemma}\label{lem:flip}
 \begin{enumerate}
   \item 
       For the shearing of the ideal quadrilateral $Q$ along $\beta$, $\textup {shr}(\beta;Q)=-\textup{shr}(\alpha;Q)$.
   \item Suppose $l_{11}\geq l_{42}$, then $\textup{shr}(\alpha;Q)>0$ and
    \begin{eqnarray*}
    |l_{i1}'-l_{i1}+\textup{shr}(\alpha;Q)|\leq \log2,&
    |l_{i2}'-l_{i1}-\textup{shr}(\alpha;Q)|\leq \log2,&\ i=1,3;\\
    |l_{i1}'-l_{i1}|\leq \log 2,&\ |l_{i2}'-l_{i2}|\leq \log 2,&\ i=2,4.
    \end{eqnarray*}
 \end{enumerate}
\end{lemma}
\begin{proof}
    Here we adopt the upper half-plane model for the hyperbolic geometry (see Fig.~\ref{fig:shearing}).
    Since the map $z\mapsto kz$ is an isometry of the hyperbolic metric, we suppose that the Euclidean coordinates of  $O_1,O_2,O_3$ are $(0,0),(1,0),(0,1+R)$ respectively.

    (1) The shearing of $Q$ along $\alpha$ is $\log R$. To calculate the shearing of $Q$ along $\beta$, we perform the fractional linear map $f:z\mapsto\frac{-z}{z-1-R}$ on the upper half plane which sends $O_1,O_2,O_3,\infty$ to $(0,0),(1/R,0),\infty,-1$ respectively . 
  It is clear in this case that the shearing of $f(Q)$ along $\tilde \beta=f(\beta)$ is $\log (1/R)$, which means the shearing of $Q$ along $\beta$ is also $\log (1/R)$ since $f$ is an orientation preserving isometry.

  (2) 
  By assumption $l_{11}\geq l_{42}$, hence $R\geq 1$.
  It is clear that
  $l_{41}+l_{42}=l_{41}'+l_{42}'$. The coordinates of $M_4,M_4'$ are $(R+1,R)$ and $(1+R,1+R)$ respectively. In addition, the coordinate of the intersection point between $\eta_4$ and $H_4$ is $(1+R,e^{l_{42}}R)$. Therefore $l_{42}'=\log {\frac{e^{l_{42}}R}{1+R}}\in [l_{42}-\log 2,l_{42}]$, hence $l_{41}'\in [l_{41},l_{41}+\log 2]$. Similarly we get $l_{22}'\in [l_{22}-\log2,l_{22}]$ and $l_{21}'\in [l_{21},l_{22}+\log2]$.

  For the remaining inequalities, note that $l_{12}=l_{21}$ and $l_{11}+l_{12}=l'_{11}+l'_{42}$, hence $l_{12}'=l_{21}'-\textup{shr}(\beta;Q)\in[l_{12}+\textup{shr}(\alpha;Q),l_{12}
  +\textup{shr}(\alpha;Q)+\log 2]$, and $l_{11}'\in[l_{11}-\textup{shr}(\alpha;Q)-\log 2,l_{11}
  -\textup{shr}(\alpha;Q)]$. Similarly we get $l_{32}'\in[l_{32}+\textup{shr}(\alpha;Q),l_{32}
  +\textup{shr}(\alpha;Q)+\log 2]$, and $l_{31}'\in[l_{31}-\textup{shr}(\alpha;Q)-\log 2,l_{31}
  -\textup{shr}(\alpha;Q)]$.
\end{proof}

The estimates from the lemma above describe a close relationship between the changes of the lengths of preferred arcs and the shearing along each simple geodesic $\eta_i$ after a diagonal-flip operation. This relationship provides a clue for investigating the effect of  the triangulation $\Gamma$ on the Hilbert metric $d^{\Gamma}_h$, i.e. investigating the relationship between $d^{\Gamma}_h(X,Y)$ and $d^{\Gamma'}_h(X,Y)$ for two different triangulations $\Gamma$, $\Gamma'$. If we fix a point $X_0\in T_{g,n}$, the sphere $B^{\Gamma}(X_0,R)$ centered at $X_0$ of radius $R$ with respect to $d^\Gamma_h$ is an almost-sphere, i.e. a sphere up to an additive constant, with respect to $d^{\Gamma'}_h$,  where $\Gamma'$ can be obtained from $\Gamma$ by a diagonal-flip.

\begin{proposition}\label{prop:diagonal-sphere}
  Fix $X_0\in T_{g,n}$. Let $\Gamma,\Gamma'$ be two preferred triangulations of $S_{g,n}$ such that one can be obtained from the other by a diagonal-flip, then there is a constant $C_{\Gamma,\Gamma',X_0,\rho_0}$ depending on $\Gamma, \Gamma, X_0, \rho_0$  such that
  $$
  |d^{\Gamma}_h(X_0,X)- d^{\Gamma'}_h(X_0,X)|\leq C_{\Gamma,\Gamma',X_0,\rho_0},
  \text{ for any } X\in T_{g,n}.
  $$
\end{proposition}
\begin{proof}
  Set $\Gamma=\{\alpha,\eta_1,...,\eta_{6g-6+2n}\}$, $\Gamma'=\{\alpha',\eta_1,...,\eta_{6g-6+2n}\}$.
  First, we prove an inequality which holds for any preferred triangulation $\Gamma$,
  \begin{equation}\label{eq:ob}
    \sup_{\eta\in \Gamma\backslash \alpha}l_{\eta}(X)
  \leq
  \sup_{\eta\in \Gamma}l_{\eta}(X)
  \leq
  2\sup_{\eta\in \Gamma\backslash \alpha}l_{\eta}(X),\text{ for any } X\in T_{g,n}.
  \end{equation}
  Indeed, there are two simple geodesics $\eta_i,\eta_j\in \Gamma$ such that  $\eta_i$, $\eta_j$, $\alpha$ bound an ideal triangle.
 Recall that
  $${l_{\eta_j}(X)-l_{\alpha}(X)+l_{\eta_i}(X)}\geq2\rho_0>0,$$
  then
  $$l_{\alpha}(X)\leq 2\max\{l_{\eta_i}(X),l_{\eta_j}(X)\}.$$
  Hence
  $$ \sup_{\eta\in \Gamma}l_{\eta}(X)
  \leq
  2\sup_{\eta\in \Gamma\backslash \alpha}l_{\eta}(X).$$
  The first inequality  is obvious. Now the proposition follows immediately from Proposition~\ref{prop:O-dist} and  (\ref{eq:ob}).
\end{proof}
\begin{proposition}\label{prop:O-dist}
  Fix $X_0\in T_{g,n}$, then there is a constant $C_{\Gamma,X_0,\rho_0}$ depending on $\Gamma,X_0,$ and $\rho_0$  such that
  $$ |d^{\Gamma}_h(X_0,X)-\frac{1}{2} \sup_{\eta\in\Gamma}\log l_{\eta}(X)|\leq C_{\Gamma,X_0,\rho_0},\text{ for any } X\in T_{g,n}$$
\end{proposition}
\begin{proof}
  Set $\Gamma=\{\eta_1,...,\eta_{6g-5+2n}\}$. By ~(\ref{eq:def2}),
   \begin{eqnarray*}
     {d}^{\Gamma}_h(X_1,X_2)&=&\frac{1}{2}(\sup_{\Delta_{ijk}\in \mathbb T_\Gamma}
     \log\frac{{l_{\eta_i}(X_1)-l_{\eta_j}(X_1)+l_{\eta_k}(X_1)}}
    {{l_{\eta_i}(X_2)-l_{\eta_j}(X_2)+l_{\eta_k}(X_2)}}\\
      &&
     +\sup_{\Delta_{ijk}\in \mathbb T_\Gamma}\log\frac{{l_{\eta_i}(X_2)
     -l_{\eta_j}(X_2)+l_{\eta_k}(X_2)}}{{l_{\eta_i}
      (X_1)-l_{\eta_j}(X_1)+l_{\eta_k}(X_1)}})\nonumber.
    \end{eqnarray*}
  Note that
  \begin{equation}\label{eq:sum}
  2l_{\eta_i}(X)= [{l_{\eta_i}(X)-l_{\eta_j}(X)+l_{\eta_k}(X)}]+
   [{l_{\eta_j}(X)-l_{\eta_k}(X)+l_{\eta_i}(X)}]
  \end{equation}
  for any ideal triangle $\Delta_{ijk}$,
   and
   $$
   {[{l_{\eta_i}(X)-l_{\eta_j}(X)+l_{\eta_k}(X)}]}\geq 2\rho_0,\text{ for any } \Delta_{ijk}.
   $$

  Therefore
  \begin{eqnarray*}
  d^{\Gamma}_h(X_0,X)
  &\leq& \frac{1}{2} [\log\frac{\sup_{\eta\in\Gamma}l_{\eta}(X)}{\rho_0}+
  \log\frac{\sup_{\eta\in\Gamma}l_{\eta}(X_0)}{\rho_0}]\\
  &=& \frac{1}{2}\sup_{\eta\in\Gamma}\log l_{\eta}(X)+\frac{1}{2}\sup_{\eta\in\Gamma}\log l_{\eta}(X_0)-\log \rho_0.
  \end{eqnarray*}
  Next we deal with the inverse inequality. Without loss of generality, we assume that $l_{\eta1}(X)=\sup_{\eta\in\Gamma}l_{\eta}(X)$ and that $\eta_1,\eta_2,\eta_3$ bounds an ideal triangle. By ~(\ref{eq:sum}), at least one of ${l_{\eta_1}(X)-l_{\eta_2}(X)+l_{\eta_3}(X)}$ and $
   {l_{\eta_2}(X)-l_{\eta_3}(X)+l_{\eta_1}(X)}$ is not less than $\sup_{\eta\in\Gamma}l_{\eta}(X)$. Then

  \begin{eqnarray*}
    d^{\Gamma}_h(X_0,X)
    &\geq& \frac {1}{2}[\log\frac{\sup_{\eta\in\Gamma}l_{\eta}(X)}{2\sup_{\eta\in\Gamma}l_{\eta}(X_0)}]\\
    &=&\frac{1}{2}\sup_{\eta\in\Gamma}\log l_{\eta}(X)-\frac{1}{2}[\log 2+\sup_{\eta\in\Gamma}l_{\eta}(X_0)].
  \end{eqnarray*}
  Set $C_{\Gamma,X_0,\rho_0}\triangleq \max\{\frac{1}{2}\sup_{\eta\in\Gamma}\log l_{\eta}(X_0)-\log \rho_0,\frac{1}{2}\sup_{\eta\in\Gamma}l_{\eta}(X_0)+\frac{1}{2}\log 2\}$, the proposition follows.
\end{proof}

\vskip 20pt
\section{Actions of Mapping class group }\label{sec:mcg-action}

The mapping class group $MCG(S_{g,n})$ consists of the isotopy classes of orientation-preserving self homeomorphisms of $S_{g,n}$. In this section, we study the actions of mapping class group $MCG(S_{g,n})$ on the metric space $(T_{g,n},d^{\Gamma}_h)$. The action is defined as following. For $g\in MCG(S_{g,n})$ and $(X,f)\in T_{g,n}$,  $g\circ(X,f)$ is defined as the  marked hyperbolic surface $(X,f\circ g^{-1})$.



 Denote by $PMCG(S_{g,n})$ the subgroup of $MCG(S_{g,n})$ consisting of elements that fix each puncture individually. It is well known that $PMCG(S_{g,n})$ can be generated by finitely many Dehn twists  about nonseparating simple closed curves, where a nonseparating simple  closed curve $\alpha$ is a closed curve such that $S_{g,n}\backslash \alpha$ is  connected (see \cite[Chap. 5]{Bu}).

The lemma below describes the changes of $l_{\eta_i}(X)$ under a Dehn twist.
\begin{lemma}\label{lem:Dehn-twist}
  Assume that $\Delta_{123}$ is an ideal triangle on $X\in T_{g,n}$ with three sides $\eta_1,\eta_2, \eta_3$, and that $g$ is a positive Dehn twist  of $S_{g,n}$ along an essential simple closed curve $\alpha$.
  Denote by $\Delta'_{123}$, $\eta'_1,\eta'_2, \eta'_3$, the images of $\Delta_{123}$ , $\eta_1,\eta_2, \eta_3$, respectively, under the action of $g$.
  Then, there is a constant $C$ depending on the length $l_\alpha(X)$, the reference height $\rho_0$ and the isotopy classes of $\alpha$,$\eta_1,\eta_2,\eta_3$   such that

  $$
  \frac{1}{C}\leq \frac{{l_{\eta'_{i+1}}(X)+l_{\eta'_{i-1}}(X)-l_{\eta'_i}(X)}}
  {l_{\eta_{i+1}}(X)+l_{\eta_{i-1}}(X)-l_{\eta_i}(X)}\leq C,\ i=1,2,3.
  $$
\end{lemma}

\begin{proof}
  Note that
  $$
  |l_{\eta'_i}(X)-l_{\eta_i}(X)|\leq i(\eta_i,\alpha)l_{\alpha}(X),\  i=1,2,3,
  $$
  $$
  i(\eta'_i,\alpha)=i(\eta_i,\alpha),\ i=1,2,3.
  $$
  On the other hand, from  ~(\ref{eq:rho}),
  $$
 {{l_{\eta'_{i+1}}(X)+l_{\eta'_{i-1}}(X)-l_{\eta'_i}(X)}}>2\rho_0,\ i=1,2,3;
  $$
  $${{l_{\eta_{i+1}}(X)+l_{\eta_{i-1}}(X)-l_{\eta_i}(X)}}>2\rho_0,  j=i-1,i,i+1.
  $$
  Hence
  \begin{eqnarray*}
  \frac{{l_{\eta'_{i+1}}(X)+l_{\eta'_{i-1}}(X)-l_{\eta'_i}(X)}}
  {l_{\eta_{i+1}}(X)+l_{\eta_{i-1}}(X)-l_{\eta_i}(X)}
  &\leq &
  1+\frac{[i(\eta_{i+1},\alpha)+i(\eta_i,\alpha)
  +i(\eta_{i-1},\alpha)]l_{\alpha}(X)}
  {l_{\eta_{i+1}}(X)+l_{\eta_{i-1}}(X)-l_{\eta_i}(X)}\\
  &\leq &
  1+\frac{[i(\eta_{i+1},\alpha)+i(\eta_i,\alpha)
  +i(\eta_{i-1},\alpha)]l_{\alpha}(X)}{2\rho_0}.
  \end{eqnarray*}
  Interchange $\eta_i$ with $\eta'_i$, we get the inverse inequality
  $$
   \frac{{l_{\eta_{i+1}}(X)+l_{\eta_{i-1}}(X)-l_{\eta_i}(X)}}
  {l_{\eta'_{i+1}}(X)+l_{\eta'_{i-1}}(X)-l_{\eta'_i}(X)}
  \leq
  1+\frac{[i(\eta_{i+1},\alpha)+i(\eta_i,\alpha)
  +i(\eta_{i-1},\alpha)]l_{\alpha}(X)}{2\rho_0}.
  $$
\end{proof}

As an application, we have the following.
\begin{corollary}\label{cor:Dehn-twist}
  Let $\Gamma$ be a preferred triangulation of $S_{g,n}$. Let $g$ be a  positive Dehn twist of $S_{g,n}$ along an essential simple closed curve $\alpha$. Set $M_{\alpha,l}\triangleq\{X\in T_{g,n}: l_{\alpha}(X)\leq l\}$.  Then there is a constant $C_{\Gamma,\alpha,\rho,l}$ depending on $\Gamma$, the isotopy class of $\alpha$ , the reference height $\rho_0$ and $l$ such that
  $$
  |d^{\Gamma}_{h}(X,Y)-d^{\Gamma}_{h}(gX,gY)|
  \leq  C_{\Gamma,\alpha,\rho_0,l},\text{ for any } X,Y\in M_{\alpha,l}.
  $$
\end{corollary}
\begin{proof}
  It follows from Lemma~\ref{lem:Dehn-twist} that
  \begin{eqnarray*}
  \frac{1}{C_{\Gamma,\alpha,\rho}}  \frac{{l_{\eta_{i+1}}(Y)+l_{\eta_{i-1}}(Y)-l_{\eta_i}(Y)}}
  {l_{\eta_{i+1}}(X)+l_{\eta_{i-1}}(X)-l_{\eta_i}(X)}
  &\leq&
  \frac{{l_{\eta_{i+1}}(gY)+l_{\eta_{i-1}}(gY)-l_{\eta_i}(gY)}}
  {l_{\eta_{i+1}}(gX)+l_{\eta_{i-1}}(gX)-l_{\eta_i}(gX)}\\
  &\leq&
  C_{\Gamma,\alpha,\rho_0}  \frac{{l_{\eta_{i+1}}(Y)+l_{\eta_{i-1}}(Y)-l_{\eta_i}(Y)}}
  {l_{\eta_{i+1}}(X)+l_{\eta_{i-1}}(X)-l_{\eta_i}(X)},
  \end{eqnarray*}
  where $C_{\Gamma,\alpha,\rho_0}=[1+(l/\rho_0)\Sigma_{i=1}^{6g-5+2n}
  i(\eta_i,\alpha)][1+(l/\rho_0)\Sigma_{i=1}^{6g-5+2n}
  i(\eta_i,\alpha)].$

\end{proof}

We do not know whether or not the action of $MCG(S_{g,n})$ on $(T_{g,n},d^\Gamma_h)$ is quasi-isometric. But for any given $X,Y\in T_{g,n}$, we have the following asymptotic behaviour.
\begin{proposition}\label{prop:aymptotic}
  Let $\Gamma$ be a preferred triangulation, and $g\in MCG(S_{g,n})$ be a positive Dehn twist about a simple closed curve $\alpha$. For any given $X,Y\in T_{g,n}$, there is a positive number $C_{X,Y}$ depending on $X,Y$ such that
  $$ \lim_{n\to\infty}d^\Gamma_h(g^nX,g^nY)=C_{X,Y}.$$
  Moreover,  for any $X\in T_{g,n}$,
  $$ \lim_{n\to\infty}d^\Gamma_h(g^nX,g^{n+1}X)=0.$$
\end{proposition}
\begin{proof}
  Note that a Dehn twist is also an earthquake map, i.e. $g=\mathcal E^{l_\alpha (X)}_\alpha $. The remaining discussion is similar to the proof of the Main Theorem.
\end{proof}

It follows immediately from Proposition \ref{prop:aymptotic} that  $(T_{g,n},d_h^\Gamma)$ is not $MCG(S_{g,n})$ invariant. More precisely, we have the following corollary.
\begin{corollary}\label{cor:nonisometry}
Let $\Gamma$ be a preferred triangulation, and $g\in MCG(S_{g,n})$ be a positive Dehn twist about a simple closed curve $\alpha$. Then the action of $g$  on $T_{g,n}$ is not isometric. In particular,  $(T_{g,n},d_h^\Gamma)$ is not $MCG(S_{g,n})$ invariant.
\end{corollary}

\end{document}